\documentclass[a4paper,12pt]{amsart}
\usepackage{amsmath,amssymb}
\usepackage{latexsym}
\usepackage{amsfonts,amscd}
\usepackage[mathscr]{eucal}
\usepackage{braket}
\usepackage[T1]{fontenc}

\usepackage{color}
\DeclareMathOperator{\dom}{Dom}

\newtheorem{proposition}{Proposition}
\newtheorem{theorem}{Theorem}
\newtheorem{lemma}{Lemma}

\theoremstyle{definition}

\newtheorem{remark}{Remark}
\newtheorem{example}{Example}
\newtheorem{definition}{Definition}

\newcommand{\R}{\mathbb R} %real
\newcommand{\C}{\mathbb C} %complex
\newcommand{\N}{\mathbb N} %natural
\newcommand{\Z}{\mathbb Z} %integer
\newcommand{\M}[1]{\ensuremath{\left\Vert \,#1\, \right\Vert}}

\newcommand{\rd}[1]{\ensuremath{\left( #1 \right)}}

\newcommand{\h}[1]{\mathcal{#1}}

\newcommand{\hi}{\mathcal{H}} %Hilbert space
\newcommand{\vi}{\mathcal{V}} %a vector subspace of a Hilbert space
\newcommand{\lh}{\mathcal{L(H)}} %bounded linear operators
\newcommand{\sv}{\mathcal{S(V)}} %sesquilinear forms
\newcommand{\ip}[2]{\left\langle\,#1|#2\,\right\rangle} %inner product
\newcommand{\ran}{\textrm{ran}} %range
\newcommand{\vp}{\varphi}
\def\<{\langle}
\def\>{\rangle}
\newcommand{\ang}[2]{\Braket{#1 | #2}}
\newcommand{\SSet}[2]{\Set{#1 \, | \, #2}}
\renewcommand{\l}{\lambda}

\newcommand{\no}[1]{\left\|#1\right\|} %norm
\newcommand{\bo}[1]{\mathcal{B}\left(#1\right)} %borel sigma-algebra

\newcommand{\E}{\mathsf{E}}
\newcommand{\F}{\mathsf{F}}
\newcommand{\Qo}{\mathsf{Q}}
\newcommand{\Po}{\mathsf{P}}
\newcommand{\LH}{\mathcal{L}\big(L^2(\R)\big)}
\newcommand{\LHH}{\mathcal{L}\big(L^2(\cal I)\big)}

\newcommand{\strgdom}[2]{D({#1},{#2})}
\newcommand{\formset}[2]{\cal W({#1},{#2})}
\newcommand{\formsetna}{\formset{f}{\E}}
\newcommand{\predom}[3]{\cal W_{#3}({#1},{#2})}
\newcommand{\predomna}[1]{\predom{f}{\E}{#1}}
\newcommand{\ga}[2]{\Gamma({#1},{#2})}
\newcommand{\choices}[2]{\cal C({#1},{#2})}
\newcommand{\weakops}[3]{\cal L_W({#1},{#2},{#3})}
\newcommand{\gac}[3]{\Gamma_c({#1},{#2},{#3})}
\newcommand{\preop}[3]{G({#1},{#2},{#3})}
\newcommand{\weakD}[3]{\hat{D}_{#3}({#1},{#2})}
\newcommand{\weakop}[3]{\hat{L}_{#3}({#1},{#2})}
\newcommand{\weaksymD}[3]{D'_{#3}({#1},{#2})}
\newcommand{\weaksymop}[3]{L'_{#3}({#1},{#2})}
\newcommand{\sqd}[2]{\tilde{D}({#1},{#2})}
 \newcommand{\sqdop}[2]{\tilde{L}({#1},{#2})}
\newcommand{\strgop}[2]{L({#1},{#2})}
\newcommand{\formdom}[2]{\cal D_F({#1},{#2})}
\newcommand{\formdomna}{\formdom{f}{\E}}

\newcommand{\mweaksymop}[2]{L'({#1},{#2})}
\newcommand{\form}[2]{F_{{#1},{#2}}}
\newcommand{\formna}{\form{f}{\E}}

\newcommand{\cal}{\mathcal}

\newcommand{\CHI}[1]{\ensuremath{ \chi\raisebox{-1ex}{$\scriptstyle #1$} }}
\makeatletter
\newlength{\PM@dpth}\newlength{\PM@hght}\newlength{\PM@wdth}
\newcommand{\clx}[1]{\settodepth{\PM@dpth}{$\underline{a}$}\settoheight{\PM@hght}{$#1$}%
\addtolength{\PM@hght}{\PM@dpth}\settowidth{\PM@wdth}{$#1$}\makebox[0pt][l]{\rule[\PM@hght]%
{\PM@wdth}{\fboxrule}}#1} \makeatother

\begin{document}
\title{Operator integrals and sesquilinear forms}
\author{D.\ A.\ Dubin}
\address{Department of Pure Mathematics, The Open University, Walton Hall, Milton Keynes, MK7 6AA, U.K.}
\email{dadubin13@gmail.com}
\author{J.\ Kiukas}
\address{Zentrum Mathematik, M5, Technische Universit\"at M\"unchen, Boltzmannstrasse 3, 85748 Garching, Germany}
\email{jukka.kiukas@tum.de}
\author{J.-P.\ Pellonp\"a\"a}
\address{Department of Physics and Astronomy, University of Turku, FI-20014 Turku, Finland}  \email{juhpello@utu.fi}
\author{K.\ Ylinen}
\address{Department of Mathematics and Statistics, University of Turku, FI-20014 Turku, Finland}
\email{ylinen@utu.fi}
\begin{abstract} We consider various systematic ways of defining unbounded operator valued integrals of complex functions with respect to (mostly)
 positive operator measures and positive sesquilinear form  measures, and investigate their relationships to each other in view of the extension theory of symmetric operators.
We demonstrate the associated mathematical subtleties with a
physically relevant example involving moment operators of the momentum observable of a particle confined to move on a bounded interval.
\end{abstract}

\maketitle

\

\noindent \begin{small} {\bf Keywords: } vector measure; operator measure; operator integral; sesquilinear form; quantum observable\end{small}

\section{Introduction}

Selfadjoint operators represent observables in the
traditional (von Neumann) description of quantum mechanics when a quantum system is associated with a Hilbert space $\mathcal{H}$.
By the spectral theorem, selfadjoint operators $A$  in $\mathcal{H}$ are in a bijective correspondence with spectral measures (normalized projection valued measures) $\mathsf{E}:\mathcal{B}\left( \mathbb{R} \right)\to
\mathcal{L(H)}$ where $\cal B(\R)$ is the Borel $\sigma$-algebra of the real line $\R$ and $\lh$ is the space of bounded operators on $\hi$.
 The correspondence in the spectral theorem can be written as an \emph{operator integral}, in the form $A=\int x\, d{\mathsf E}(x)$. More specifically, if $\E:\bo\R\to \lh$ is a normalized projection valued measure, and $f:\R\to \R$ a Borel
measurable (possibly unbounded) function, there exists a unique operator, denoted $\int f\, d\E$, such that its domain
\begin{equation}\label{sqint}
\dom\left(\int f\, d\E\right) = \left\{ \vp\in \hi \,\bigg| \int |f(x)|^2\, d\E_{\vp,\vp}(x) <\infty\right\}
\end{equation}
is dense, and, for all $\psi\in \hi$, $\vp\in \dom(\int f \, d\E)$,
\begin{equation}\label{opintegral}
\left\langle \psi\,\bigg|\left(\int f\, d\E\right)\vp\right\rangle = \int f(x)\, d\E_{\psi,\vp}(x)
\end{equation}
where $\E_{\psi,\vp}(X):=\langle\psi|\E(X)\vp\rangle$, $X\in\cal B(\R)$.
 This operator is selfadjoint and
 its spectral measure is
$X\mapsto \E\big(f^{-1}(X)\big)$. In addition, $\|(\int f \, d\E)\vp\|^2 = \int |f(x)|^2 \, d\E_{\vp,\vp}(x)$, consistent with  the feature
that the domain  consists of exactly those vectors for which the integral of the \emph{square} of $f$ is finite.

However, from the operational point of view of quantum measurement theory, this definition is often considered too restrictive: in standard modern quantum theory (in particular, quantum information theory), a generalization to (normalized) positive operator (valued) measures is used instead.
 A physical consequence is that a positive operator measure (POVM) which is not a projection valued measure (PVM) will, in particular, allow some imperfections of measurement.

Going from projection valued measures to general positive operator measures, some useful features of the theory are lost, most notably the spectral theorem and functional calculus. However, some ideas of spectral theory may be retained: According to Naimark's dilation theory, as given e.g.\ in \cite{riesz} or \cite{Akhiezer}, for any densely defined
symmetric operator $A$ in $\hi$ there exists a normalized POVM $\E:\h B(\R)\to \h \lh $, having the properties
\begin{equation}\label{symexpect}
\langle \psi | A\vp\rangle = \int x \, d\E_{\psi,\vp}(x), \; \; \psi\in \hi, \; \vp\in \dom(A),
\end{equation}
and
\begin{equation}\label{symnorm}
\| A\vp\|^2 = \int x^2 \, d\E_{\vp,\vp}(x), \; \; \vp\in \dom(A).
\end{equation}
However, unlike the case of spectral measures, the domain of $A$ need not coincide with the set of vectors for which the integral in \eqref{symnorm} is finite.
Moreover, the correspondence does not work the other way: not every POVM $\E:\h B(\R)\to \h \lh $
satisfies \eqref{symexpect} and \eqref{symnorm} with respect to some symmetric operator $A$. This has been noted in the
above references, along with the fact that the integral in the right hand side of \eqref{symnorm} may well be infinite for
\emph{any} nonzero vector $\vp$. Moreover, a normalized POVM corresponding to a symmetric operator $A$ as above is unique only if $A$ is
maximally symmetric (i.e.\ has no proper symmetric extension) \cite{Akhiezer}.

For these reasons, going from a POVM to a symmetric operator is not straightforward and choosing  a reasonable definition for the operator integral
$\int f \, d\E$ (including its domain) is problematic -- except when $f$ is bounded, in which case the domain is all of $\hi$.

In fact, the difficulties in choosing the domain have led the authors in \cite[p.\ 132]{Akhiezer} to consider $\int x \, d\E(x)$ in a symbolic sense only, as a shorthand for the
equations \eqref{symexpect} and \eqref{symnorm}, provided they hold for the given POVM.
As pointed out by Werner \cite{WernerScreen}, however, even the general
operator integral $\int f\, d\E$ can  be uniquely defined as a symmetric operator on the domain \eqref{sqint},
 so that
\eqref{opintegral} holds, in contrast to what appears to be intended in \cite[p.\ 132]{Akhiezer}.
(See the above paper by Werner, and also \cite{Lahti}.)
 The reason why this does not contradict the
observation that not every POVM satisfies \eqref{symexpect} and \eqref{symnorm} for some symmetric operator, is simply that
\eqref{symnorm} does not hold for $A= \int x\, d\E(x)$, in general.

When \eqref{symnorm} holds, with \eqref{sqint} dense, the POVM
is called \emph{variance free} \cite{Werner}. For a general POVM it may be the case that only the inequality
\begin{equation}\label{opintineq}
\left\|\left(\int f\, d\E\right)\vp\right\|^2\leq \int |f(x)|^2\, d\E_{\vp,\vp}
\end{equation}
holds. The domain of \eqref{sqint} has a physical meaning as the set of those
vector states for which the measurement distribution has finite variance. For this reason, this set
is a natural domain for the \emph{variance form}
$$(\psi,\vp)\mapsto \int x^2\, d\E_{\psi,\varphi}-\langle\tilde{\E}[1]\psi|\tilde{\E}[1]\varphi\rangle\in \C$$
 where $\tilde{\E}[1]=\int x\, d\E$ is the first moment operator of $\E$ (see Section \ref{viimeinenluku}).
This definition for the domain of the operator integral appears most frequently in the literature, see e.g.\ \cite{WernerScreen,Schroeck,Akhiezer}.

One might think that above the definition would settle the question of defining the operator integral. However, after losing the
equality in \eqref{opintineq}, it is no longer clear whether the finiteness of the integral in the right hand side is actually
\emph{needed} to define the operator integral. Loosely speaking, the reason for the square of $f$ appearing in the definition of the domain is connected to the multiplicativity of the projection valued measure, which is no longer true for POVMs. In fact, \emph{the square integrability domain \eqref{sqint} is
not necessarily the largest possible one where \eqref{opintegral} defines an operator.} This is easy to see: for example,
consider the POVM $X\mapsto \E(X):=\mu(X)I$, where $\mu$ is a probability measure and $I$ the identity operator on any Hilbert
space. If $f$ is a $\mu$-integrable function, the integrals $\int f\, d\E_{\psi,\vp}= \langle \psi |(\int f\, d\mu)\vp\rangle$
determine a well-defined operator with domain all of $\hi$, even if $\int |f|^2\, d\mu =\infty$, collapsing the domain
\eqref{sqint} to $\{0\}$. Hence, the natural definition of an operator integral needs closer mathematical examination.

A different definition has been used in \cite{Lahti,LahtiII}; we call this the \textit{strong} operator integral. As we will see, even this choice is not the largest reasonable, and we will also define \textit{weak} operator integrals which have still larger domains than the strong one. These are more operationally motivated, as they are constructed from the scalar measures $X\mapsto \langle \psi |\E(X)\varphi\rangle$.

The structure of the paper is as follows. We begin by considering strong operator integrals in the setting of general Banach spaces.
 When specializing to Hilbert spaces and positive operator measures,
 the role of the square integrability domain is explained.   Subsequently, we proceed to introduce weak operator integrals, and investigate their connection to operators defined via quadratic forms. A physically motivated example concludes the paper.

\section{Preliminaries and notations}

We begin with a fairly general setting: let $E$ and $F$ be Banach spaces and $L(E,F)$ the space of bounded linear operators
$T:E\to F$. (We use complex scalars as our main applications deal with complex Hilbert spaces.) Consider a measurable space $(\Omega,\mathcal A)$
(where by definition $\mathcal A$ is a $\sigma$-algebra of subsets of $\Omega$). A  map
$M:\mathcal A\to L(E,F)$ is called an {\em operator measure} if it is strong operator (or briefly, strongly)
$\sigma$-additive. This means that for each $x\in E$ the map $X\mapsto M_x(X):=M(X)x$ is a {\em vector measure}, i.e. $\sigma$-additive with respect to the norm in
$F$. By the Orlicz-Pettis theorem it is equivalent to require that for any $x\in E$ and $y'\in F'$ (the topological dual of $F$), the function
$X\mapsto M_{y',x}(X):=\langle y',M(X)x\rangle $ on $\mathcal A$ is a complex measure.
The following definition agrees with the usage in \cite{Dunford}.
 (We only integrate $\mathcal A$-measurable functions, though this restriction could be relaxed somewhat, see e.g. \cite{Ylinen}.)

\begin{definition}\label{def:vectorintegral}\rm  Let $\mu:\mathcal A\to F$ be a vector measure
 and $f:\Omega\to \C$ an $\mathcal A$-measurable function.
The function $f$ is {\em $\mu$-integrable} if there is a sequence $(f_n)$  of simple functions converging to $f$ pointwise
and such that $\lim_{n\to\infty}\int_Xf_n\,d\mu$ exists for all $X\in\mathcal A$.
Then $\int_\Omega f\,d\mu:=\lim_{n\to\infty}\int_\Omega f_n\,d\mu$
is called the
{\em integral} of f with respect to $\mu$.
\end{definition}

\begin{remark}\label{rem:lewis}\rm It turns out to be equivalent to the above definition to require that
$f$ is integrable with respect to the complex measure $\mu_{y'}:=y'\circ \mu$ for every $y'\in F'$
and for each $X\in\mathcal A$ one has $\nu(X)\in F$ (clearly unique) such that $\langle y',\nu(X)\rangle =\int_Xf\,d\mu_{y'}$
for all $X\in \mathcal A$, $y'\in F'$. (See \cite{Lewis70}, and \cite{Ylinen} for another proof.) If $f$ is integrable
with respect to every $\mu_{y'}$, it follows from the dominated convergence theorem and the uniform boundedness principle
(as in e.g. \cite[p. 328]{Lahti}) that for each $X\in\mathcal A$ there is some $\nu(X)\in F''$ satisfying
$\langle y',\nu(X)\rangle =\int_Xf\,d\mu_{y'}$ for each $y'\in F'$, and so in case $F$ is reflexive, we can conclude that
$f$ is actually $\mu$-integrable. We use this observation especially when $F$ is a Hilbert space.
\end{remark}

Let $\hi$ be a (complex) Hilbert space, and let $\lh$ denote the space of bounded operators on
$\mathcal{H}$. We do not have to assume that $\hi$ is separable, except in some examples where this is clearly indicated.
The identity operator of $\hi$ is denoted by $I_\hi$ or simply by $I$.
 For $\psi,\varphi\in \mathcal{H}$, we use
the symbol $\Ket{\psi}\Bra{\varphi}$ to denote the rank one
operator $\eta\mapsto \ang{\varphi}{\eta}\,\psi$. For a (linear) operator $A$ in
$\mathcal{H}$, we let $\dom(A)$ denote the domain of $A$, i.e.\
the (linear) subspace of $\mathcal{H}$ on which $A$
is defined.
As before, $(\Omega,\mathcal A)$ is a measurable space.
We let $\mathcal{B}\left( \Omega \right)$ denote
the Borel $\sigma$-algebra of any topological space $\Omega$.
We follow the convention $\N=\{0,1,2,\ldots\}$, and let $\CHI X$ be the characteristic function of the set $X\in\cal A$.

\begin{definition}\label{def:POVM}\rm Let $\mathsf{E}:{\cal A} \to \mathcal{L(H)}$ be a function.
\begin{enumerate}
\renewcommand{\labelenumi}{(\alph{enumi})}
\item
$\mathsf{E}$ is a \emph{positive operator (valued) measure}, or POVM
for short, if $\mathsf{E}$ is an operator measure and $\E(X)\ge0$ for all $X\in
\cal A$.
\item A POVM $\mathsf{E}$ is \emph{normalized} if
$\mathsf{E}(\Omega)=I$.
\item A projection valued POVM (PVM for short) which is normalized is a \emph{spectral
measure}.
\end{enumerate}
\end{definition}
\noindent
For a POVM $\mathsf{E}:\mathcal{A}\to \mathcal{L(H)}$
and $\psi,\varphi\in \mathcal{H}$, we let
$\mathsf{E}_{\psi,\varphi}$ denote the complex measure
$X\mapsto \ang{\psi}{\mathsf{E}(X)\varphi}$ and
$\mathsf{E}_\varphi$ denote the $\mathcal{H}$-valued vector
measure $X\mapsto \mathsf{E}(X)\varphi$.

\emph{Naimark's dilation theorem} (see e.g.\ \cite{riesz}) states that, for any  POVM $\E:\mathcal{A}\to \lh$, there
exists another Hilbert space $\h K$, a spectral measure $\F:\mathcal{A}\to \h L (\h K)$, and a bounded linear map $V:\hi \to \h K$,
such that $\E(X) = V^*\F(X)V$ for all $X\in \mathcal{A}$. If the set of the linear combinations of vectors $\F(X)V\varphi$, $X\in\mathcal A$, $\varphi\in\hi$, is dense in $\h K$, then the Naimark dilation $(\cal K,\F,V)$ is said to be {\it minimal}.
Note that $\E$ is normalized if and only if $V$ is an isometry, i.e.\ $V^*V=I$. In that case, $\hi$ can be identified with the range of $V$, a subspace of $\cal K$.

We already discussed integration with respect to a vector measure. Since an operator measure usually fails to be norm $\sigma$-additive,
integration with respect to operator measures needs a different approach.
For a bounded measurable function $f:\Omega\to \C$, integration with respect to a POVM is, however, quite elementary
(see e.g. \cite{Berberian}).
One starts by  setting
$
\int f d\mathsf{E} := \sum_{n=0}^k c_n
\mathsf{E}(X_n)\;\text{for}\; f= \sum_{n=0}^k c_n \CHI{X_n},\; c_n\in\C,
\; X_n\cap X_m=\emptyset, \; m\ne n\in\N.
$
The extension from these simple functions to bounded $\cal A$-measurable
functions $f:\Omega\to \mathbb{C}$ requires the convergence  of
the integrals of simple functions forming a uniformly convergent sequence. Ultimately this depends
on the fact that the range of any POVM is norm bounded, and the
resulting integral defines a bounded operator.

The following lemma is straightforward to prove by using the usual approximation techniques appearing in the construction of the integral.
\begin{lemma}\label{naimarkintegral} Let $(\h K,\F,V)$ be a Naimark dilation of $\E$. Then for every bounded $\h A$-measurable function $f:\Omega\to \C$, we have
$\int f d\E = V^*\left(\int f d\F\right) V$.
\end{lemma}

For unbounded functions,  even
defining a domain for the operator valued integral needs attention.
We study this question next.

\section{Strong operator integrals}

 Let $(\Omega,\mathcal A)$ be a measurable space. We first consider general Banach spaces $E$ and $F$.

\begin{definition}\label{strongintBan}\rm Let $M:\mathcal A\to L(E,F)$ be an operator measure and $f:\Omega\to\C$  an $\mathcal A$-measurable
function. We let $D(f,M)$ denote the subset of $E$ consisting of those $x\in E$ for which $f$ is integrable with respect to the vector measure
$X\mapsto M_x(X)=M(X)x$.
If $x\in D(f,M)$, we denote by $L(f,M)x$ the integral $\int_{\Omega} f dM_x$.
\end{definition}
\begin{proposition}\label{IntLin} If $f:\Omega\to\C$ is  an $\mathcal A$-measurable function, the set $D(f,M)$, the domain of $L(f,M)$,
is a vector subspace of $E$,
 and $L(f,M): D(f,M)\to F$ is a linear map.
\end{proposition}
\begin{proof} See e.g. \cite{Ylinen}, Corollary 3.7.
\end{proof}

 The following proposition is an immediate consequence of Remark \ref{rem:lewis}.

\begin{proposition}\label{strongintRefl} Assume that the Banach space  $F$ is reflexive.  For $x\in E$ the following conditions are equivalent:

{\rm(i)} $x\in D(f,M)$;

{\rm(ii)} $f$ is $M_{y',x}$ integrable for a all $y'\in F'$.
\end{proposition}

We mainly apply the above results in the case where $F=\hi$, a Hilbert space.

\begin{definition}\label{def:ortho} \rm We say that a vector measure $\mu:\mathcal A\to \hi$ is {\em orthogonally scattered\,}
if $$\ip{\mu(X)}{\mu(Y)}=0$$ whenever the sets $X,\,Y\in\mathcal A$ are disjoint.
\end{definition}

Orthogonally scattered vector measures have a highly developed theory, see e.g. \cite{Masani68}.
A basic observation is that
if  $\mu:\mathcal A\to \hi$ is an orthogonally scattered vector measure, by denoting  $\lambda(X)=\lambda_\mu(X):={\no{\mu(X)}}^2$,
we get a finite positive measure $\lambda$ on $\mathcal A$. The following result is well known and we only give a brief 
indication of proof.

\begin{proposition}\label{prop:orthoint}\rm  Let $\mu:\mathcal A\to\hi$ be a an orthogonally scattered vector measure and $\lambda=\lambda_\mu$ the positive measure defined above. An $\mathcal A$-measurable function $f:\Omega\to\C$ is $\mu$-integrable if and only if $|f|^2$ is $\lambda$-integrable, in which case
$\no{\int_\Omega f\,d\mu}^2=\int_\Omega |f|^2 d\lambda.$
\end{proposition}
\begin{proof}  In one direction, one may use the argument in the proof of Lemma A.2 (b) in \cite{Lahti}. In the other direction
a technique from the proof of Proposition \ref{prop:squareint} below may be adapted.
\end{proof}

\begin{remark}\label{rem:operortho}\rm (a) It follows from the above proposition that if $E$ is a Banach space and $M:\mathcal A \to L(E,\hi)$ is an operator measure such that for each $x\in E$ the vector measure $M_x:\mathcal A\to \hi$ is orthogonally scattered, then the domain
$D(f,M)$ of the strong operator integral $L(f,M)$ consists of precisely those vectors $x\in E$ for which
$|f|^2$ is integrable with respect to the measure $X\mapsto {\no{M_x(X)}}^2$ on $\mathcal A$.

(b) If $\E:\mathcal A\to\lh$ is a PVM, then for each $\varphi\in\hi$ the vector measure $\E_\varphi$ is orthogonally scattered and  ${\no{\E_\varphi(X)}}^2=\ip{\varphi}{\E(X)\varphi}$ whenever $X\in \mathcal A$.

(c) Consider the Hilbert space $\ell^2=\ell^2(\N)$. Let $\mathcal A$ be the power set of $\N$. Let
$g:\N\to \C$ be a bounded function and define $M:\mathcal A\to \lh$  by the formula  $M(X)\varphi=g\varphi\CHI X$ for all $X\in\mathcal A$, $\varphi\in\ell^2$.
Then $M$ satisfies the assumption in (a), so that $D(f,M)$ consists of those $\varphi\in \ell^2$ for which $fg\phi\in\ell^2$.
Note that  $M$ need not be a PVM, nor even a POVM. This example can be easily extended for more general measure spaces. 
\end{remark}
 We have seen (the well-known fact) that for a PVM $\E:\mathcal A\to\lh$, a vector $\varphi\in\hi$ belongs to $D(f,\E)$
 if and only if $|f|^2$ is integrable with respect to the measure $\E_{\varphi,\varphi}$. More generally, for any POVM
 $\E:\mathcal A\to\lh$ we call the set $\sqd{f}{\E}:=\{\varphi\in\hi\,|\, |f|^2 \text{ is $\E_{\varphi,\varphi}$-integrable}\}$
 the {\em square integrability domain} for the integral $\int_\Omega f\,d\E$. This makes sense, as it is known that
 $\sqd{f}{\E}$ is a linear subspace of $\hi$ contained in $D(f,\E)$. In \cite{Lahti} this was given a direct elementary proof.
 The authors of \cite{Lahti} were unaware that this result essentially had already appeared in \cite{WernerScreen}, where the proof
 is based on Naimark's dilation theorem. (For completeness, we give a proof below reproducing the idea in \cite{WernerScreen}.) The fact that $\sqd{f}{\E}$ is a linear subspace is implied by the following
easy consequence of the Cauchy-Schwarz inequality. We state it explicitly as it will also have some later use. (The terminology will be recalled at the beginning of Section \ref{weakopint}.)

\begin{lemma}\label{formineq}
Let $\cal V$ be a vector space, and $q:\cal V\times \cal V\to \C$ a positive sesquilinear form. Then
$$
q(\varphi+\psi,\varphi+\psi)\leq 2q(\varphi,\varphi)+ 2q(\psi,\psi),\qquad
\varphi,\psi\in \cal V.
$$
\end{lemma}
 
\begin{proposition}\label{prop:squareint}
The vector valued integral $\int f d\E_\varphi$ exists for each $\varphi\in \sqd{f}{\E}$.
\end{proposition}
\begin{proof}
Let $(\h K,\F,V)$ be a Naimark dilation of $\E$ and $(f_n)$ a sequence of simple functions converging pointwise to $f$ with $|f_n|\leq |f|$. Then the bounded operator $\int f_n d\E$ is defined for each $n$ according to the definition in the preceding section. Fix $\varphi\in \sqd{f}{\E}$. Using Lemma \ref{naimarkintegral}, and the multiplicativity of the spectral measure $\F$, we have for each $X\in \cal A$, that
\begin{align*}
\Big\|\int_X (f_n-f_m) d\E_\varphi\Big\|^2 &= \Big\|V^*\int_X (f_n-f_m) d\F V\varphi\Big\|^2\leq \|V^*\|^2 \Big\|\int_X (f_n-f_m) d\F V\varphi\Big\|^2\\
&= \|V^*\|^2 \int_X |f_n-f_m|^2 d\F_{V\varphi,V\varphi}=\|V^*\|^2\int_X |f_n-f_m|^2 d\E_{\varphi,\varphi}.
\end{align*}
Since $|f|^2$ is integrable, it thus follows from the dominated convergence theorem that the sequence $(\int_X f_n d\E_\varphi)$ of vectors is a Cauchy sequence, and thus converges. This proves the existence of the integral $\int f d\E_{\varphi}$ of $f$ with respect to the vector valued measure $\E_\varphi$.
\end{proof}
According to this result, we can define a linear operator
\begin{align*}
\sqdop{f}{\E}:\, \sqd{f}{\E}\to \hi, \qquad \sqdop{f}{\E}\varphi := \int f d\E_\varphi.
\end{align*}
Since $f$ is integrable with respect to each scalar measures $\E_{\psi,\varphi}$, $\psi\in \hi$, if is integrable with respect to $\E_\varphi$ (see e.g. \cite{Dunford}), it follows that
$
\langle \psi |\sqdop{f}{\E}\varphi\rangle = \int f d\E_{\psi,\varphi}\text{ for all } \varphi\in \sqd{f}{\E},\, \psi\in \hi.
$
Since $\sqd{f}{\E}=\{\vp\in \hi \mid V\vp\in \sqd{f}{\F}\}$, where $\sqd{f}{\F}$ is the domain of the selfadjoint operator $\sqdop{f}{\F}$, it now follows that
\begin{equation}\label{squarenaimark}
\sqdop{f}{\E} = V^*\sqdop{f}{\F}V.
\end{equation}
(see also \cite{LahtiII}.)

Summarizing, for a POVM $\E$ and a measurable function $f$, we have $\tilde D(f,\E)\subset D(f,\E)$ and $\tilde L(f,\E)\subset L(f,\E)$
where
\begin{align}\label{domscal}
\strgdom{f}{\E}&=\Big\{ \varphi\in \hi\Big| \int |f|\,d|\E_{\psi,\varphi}|<\infty \text{ for all }\psi\in \hi\Big\},\\
\langle \psi|\strgop{f}{\E}\varphi\rangle &= \int f d\E_{\psi,\varphi},
\end{align}
for the  total variation $|\E_{\psi,\varphi}|$ of $\E_{\psi,\varphi}$.

In definition \ref{strongintBan} we used the notation $L(f,M)$ but did not give it a name. From now  on, we call it {\em the}
strong operator integral of $f$ or the {\em maximal} strong operator integral of $f$ with respect to the operator measure $M$. 
If $D$ is a linear subspace of $D(f,M)$, we may call the restriction of $L(f,M)$ to $D$ {\em a} strong operator integral. 
Thus for a POVM $\E$, the operator $\tilde L(f,\E)$ is a strong operator integral. In this Hilbert space setting the key to our terminology
is the possibility to use the whole of $\hi$ as a ``test space'': for any $\varphi$ in the appropriate domain, the integral of $f$ 
with respect to $\E_{\psi,\phi}$ for {\em every} $\psi\in\hi$ exists.

\begin{example}\label{ex:strongtestspace}\rm Let $A$ be an unbounded selfadjoint operator in $\hi$ and $\E:\mathcal B(\R)\to\lh$ its spectral measure.
Then $A=L(f,\E)$ and $D(f,\E)=\{\varphi\in\hi\,|\, f \text{ is $\E_{\psi,\varphi}$-integrable
for all $\psi\in\hi$ }\}$ where $f:\R\to\R$ is the identity map. 
Since $\E_{\psi,\varphi}(X)=\overline{\E_{\varphi,\psi}(X)}$, we may also observe that if $\psi\in\hi$,
then $f$ is $\E_{\varphi,\psi}$-integrable for all $\varphi$ in the dense subspace $D(f,\E)$ of $\hi$. But this does not imply that
$D(f,\E)=\hi$. In particular, we see that in Proposition \ref{strongintRefl} it is not enough to assume the $M_{y',x}$-integrability of 
$f$  for all
$y'$ in a dense subspace of $F'$.
\end{example}

The above example may serve as a motivation for considering integration with respect to operator measures where the requirement for 
the test space described before the example is relaxed. This leads us to a host of possibilities for so-called weak operator integrals
whose analysis will be our main concern in the sequel.

\section{Weak operator integrals}\label{weakopint}

Often in physical applications one is led to consider the scalar measures $X\mapsto \E_{\psi,\vp}(X)=\langle \psi |\E(X)\vp\rangle$
related to a Hilbert operator measure $\E$ instead of the vector measures $\E_\vp$. In this section we set up a very general framework for this.
For any vector spaces $\vi_1$, $\vi_2$, a map $S:\,\vi_1\times\vi_2\to\C$ is said to be a {\em sesquilinear form}, or just sesquilinear, if it is
linear in the second and antilinear  (i.e. conjugate linear) in the first argument. Such an $S$ is {\em positive} if $\vi_1=\vi_2$
and $S(\vp,\vp)\geq 0$ for all $\vp\in \vi_1$. Then $S$  satisfies
$ S(\psi,\vp)=\overline{S(\vp,\psi)}$ for all $\psi,\,\vp\in\vi_1$.

Any vector space $\vi$ may be regarded as a dense linear subspace of a Hilbert space $\hi$: take $\hi=\ell^2_K$ where $K$ is a Hamel
basis of $\vi$. In the context of sesquilinear forms there is, however, often  a postulated way the vector space is embedded as a dense
subspace of a Hilbert space. When this is the case, it is clear from the context so that, for example, there is a  given norm and hence a topology on on $\vi$.

Let $\vi\subseteq\hi$ be a dense (linear) subspace of $\hi$ and
$\sv$ the vector space of sesquilinear
forms $S:\,\vi\times\vi\to\C$.
Assume that $\E:\,{\cal A}\to\sv$ is a {\it positive sesquilinear form valued measure,} i.e.\
\begin{enumerate}
\item[(a)] $\E_{\psi,\vp}:\,\cal A\to\C,\;X\mapsto\E_{\psi,\vp}(X):=[\E(X)](\psi,\vp)$, is a complex measure for all $\psi,\,\vp\in\vi$,
\item[(b)] $\E_{\vp,\vp}(X)\ge 0$ for all $\vp\in\vi$ and $X\in\cal A$.
\end{enumerate}
We refer the reader to
\cite{HyPeYl1,HyPeYl2}
for a detailed study of such measures.
Note that any POVM $\E':\,\cal A\to\lh$ defines a unique positive sesquilinear form valued measure $\E:\,\cal A\to\cal S(\hi)$ by setting $[\E(X)](\psi,\vp):=\langle\psi|\E'(X)\vp\rangle$ (thus, in the case of POVMs, we may put $\vi=\hi$ below). We always identify $\E'$ with $\E$ and by an abuse of notation simply write $\E'=\E$.
Throughout this section, $f:\cal A\to\C$ is an $\cal A$-measurable function.

\subsection{Definition}

We begin with the maximal set of pairs $(\psi,\vp)$ for which $\int f \, d\E_{\psi,\vp}$ makes sense:
$$
\formset{f}{\E} := \{ (\psi,\vp) \in \vi\times \vi \mid f \text{ is } \E_{\psi,\vp}\text{--integrable}\}.
$$
Note that $\overline{\E_{\psi,\vp}(X)}\equiv\E_{\vp,\psi}(X)$ by positivity
so that $|\E_{\psi,\vp}|=|\E_{\vp,\psi}|$ and, hence,
$\formset{f}{\E}\subseteq\vi\times \vi$ is symmetric, i.e.\
$ (\psi,\vp) \in \formset{f}{\E}$ implies $ (\vp,\psi) \in \formset{f}{\E}$.
We then put
\begin{equation}\label{predomain}
\predomna{\varphi}:= \{ \psi\in \vi \mid (\psi,\vp)\in \formsetna \}
\end{equation}
for each $\vp\in \vi$. Since
$\E_{\alpha \psi_1+\beta\psi_2,\vp}= \overline\alpha \E_{\psi_1,\vp} +\overline\beta \E_{\psi_2,\vp}$, $\alpha, \beta\in \C$, $\psi_1,\psi_2\in \vi$, it follows that each $\predomna{\vp}\subseteq\vi\subseteq\hi$ is a linear subspace, and the
functional $\psi\mapsto \overline{\int f d\E_{\psi,\varphi}}$ is linear on that subspace. A similar argument shows that
\begin{equation}\label{predomincl}
\predomna{\varphi_1}\cap\predomna{\varphi_2}\subseteq \predomna{\alpha \vp_1+\beta \vp_2}
\end{equation}
for any $\vp_1,\vp_2\in \vi$ and $\alpha ,\beta \in \C$.

We are now interested in (linear) {\it operators} $T:\,\dom(T)\to\hi$ determined by these integrals through $\langle \psi |T\vp\rangle=\int f \, d\E_{\psi,\vp}$. Accordingly, such an operator should have the property that for each $\vp\in\dom(T)$:
$\langle \psi |T\vp\rangle=\int f \, d\E_{\psi,\vp}$, where $\psi$ runs through some subset $\cal S_\vp$ of $\predomna{\varphi}$ which separates the points of $\hi$ in the usual sense of self-duality of $\hi$.
We make this separation requirement  to always guarantee that the vector $T\vp$ is uniquely determined by the integrals $\int f \, d\E_{\psi,\vp}$ via the inner products just mentioned.
Note that here we really want to determine $T\vp$, and the vector $\psi$ is just in an auxiliary role.\footnote{If we would be interested in \emph{sesquilinear forms} rather than operators, then we should consider $\psi$ and $\vp$ in an equal footing. However, here we want to consider \emph{operator} integrals, so the given requirement is clearly the most natural one.} Since each $\predomna{\vp}$ is a linear subspace, the necessarily dense\footnote{Note that the orthogonal complement of a separating subset $\cal S$ is $\hi$, so $\cal S$ generates a dense subspace.} linear span $\cal D_\vp$ of such a separating subset $\cal S_\vp$ is also included in $\predomna{\varphi}$, and by linearity,
$\langle \psi |T\vp\rangle = \int f \, d\E_{\psi,\vp}$ for all $\psi\in\cal D_\vp$. Hence, we can take the separating subsets to be dense subspaces without restricting generality.

The above requirements imply, in particular, that $\dom(T)$ must be a subset of $$\ga{f}{\E}:=\{ \vp\in \vi\mid \predomna{\vp} \text{ is dense in }\hi\}.$$ The requirement of choosing the separating subspaces can now be formulated as follows:
Let $\choices{f}{\E}$ denote the family of maps
\begin{align*}
&\Phi:\ga{f}{\E}\to \{\cal D\subseteq \hi \mid \cal D \text{ is a dense subspace}\},\\ &\Phi(\vp)\subseteq \predomna{\vp} \text{ for all }\vp\in \ga{f}{\E}.
\end{align*}
Note that $\choices{f}{\E}\neq \emptyset$, because an obvious choice is $\Phi(\vp)=\predomna{\vp}$ for all $\vp\in \ga{f}{\E}$. We can now state the definition of a weak operator integral.
\begin{definition}\rm
We say that a linear operator $T:\dom(T)\to \hi$ is a \emph{weak operator integral} of $f$ with respect to $\E$, if $\dom(T)\subseteq \ga{f}{\E}$, and there exists a map $\Phi\in \choices{f}{\E}$, such that
\begin{equation}\label{weak}
\langle \psi |T\vp\rangle = \int f \, d\E_{\psi,\vp}, \quad \text{ for all } \vp\in \dom(T), \, \psi\in \Phi(\vp).
\end{equation}
We then also say that the weak operator integral $T$ is \emph{associated with the map $\Phi$}. For each $\Phi\in \choices{f}{\E}$, we let $\weakops{f}{\E}{\Phi}$ denote the set of weak operator integrals associated with $\Phi$.
\end{definition}
Note that $\ga{f}{\E}$ always contains at least the trivial subspace $\cal D_0 =\{0\}$, so for every choice of $\Phi$ there corresponds at least a trivial weak operator integral.

The choice of the function $\Phi$ is crucial; different choices may correspond to different operators $T$, because on the one hand, dense subspaces can even have trivial intersection, see Section \ref{viimeinenluku} for an example, and on the other hand, different choices can lead to the same operator.
In particular, we have the following result.

\begin{proposition} Let $\E:\cal A\to\lh$ be a POVM.
Each strong operator integral is also a weak operator integral associated with every $\Phi\in \choices{f}{\E}$.
\end{proposition}
\begin{proof} According to \eqref{domscal}, the domain of the maximal strong operator integral is given by
$$
\strgdom{f}{\E}=\{\varphi\in \hi\mid \predomna{\vp}=\hi\},
$$
so $\strgdom{f}{\E}\subseteq \ga{f}{\E}$. Given any $\Phi\in \choices{f}{\E}$, equation \eqref{weak} holds because of \eqref{domscal}.
\end{proof}

Now, given a map $\Phi\in \choices{f}{\E}$, we set
$$
\gac{f}{\E}{\Phi}:=\left\{ \varphi\in \ga{f}{\E}\Big| \Phi(\vp)\ni \psi\mapsto \int f\, d\E_{\psi,\vp} \text{ is continuous} \right\},
$$
and use the Frechet-Riesz theorem to define a unique map
\begin{align*}
\preop{f}{\E}{\Phi}:\,\gac{f}{\E}{\Phi}\to \hi, \qquad \langle \psi| \preop{f}{\E}{\Phi}\vp\rangle&=\int f d\E_{\psi,\vp} \, \text{ for all }\psi\in \Phi(\vp).
\end{align*}
Clearly, the domain of any weak operator integral associated with the map $\Phi$ is included in $\gac{f}{\E}{\Phi}$. This observation immediately gives the following characterization.
\begin{proposition}\label{weakconstr} Fix a $\Phi\in \choices{f}{\E}$. Given any subspace $\cal D_0$ of $\hi$, which is included in $\gac{f}{\E}{\Phi}$, the restriction $\preop{f}{\E}{\Phi}|_{\cal D_0}$ is a weak operator integral (with domain $\cal D_0$) associated to $\Phi$. Conversely, every element of $\weakops{f}{\E}{\Phi}$ is obtained this way.
\end{proposition}

Since the intersection of two dense subspaces does not have to be dense (it can even be $\{0\}$), it is clear that $\ga{f}{\E}$, and therefore also $\gac{f}{\E}{\Phi}$ are not themselves linear subspaces, in general. Hence, there is no canonical choice for a maximal weak operator integral associated with a given map $\Phi$. However, it follows immediately from the above proposition that given two operators $T,T'$, such that $T'\subseteq T$ (that is, $\dom(T')\subseteq\dom(T)$ and $T'\vp=T\vp$, $\vp\in\dom(T')$), and $T\in \weakops{f}{\E}{\Phi}$, it follows that $T'\in \weakops{f}{\E}{\Phi}$. In particular, the (nonempty) set $\weakops{f}{\E}{\Phi}$ is partially ordered via the usual operator ordering, or, equivalently, the inclusion of domains. Moreover, every (nonempty) totally ordered subset of $\weakops{f}{\E}{\Phi}$ has an upper bound in $\weakops{f}{\E}{\Phi}$ (the upper bound is obtained by taking the union of the domains of the operators in the chain). Hence, by Zorn's lemma, there exists at least one maximal element in $\weakops{f}{\E}{\Phi}$. We call such an element a \emph{maximal weak operator integral} associated to $\Phi$.

\begin{example}\label{bounded} For a POVM $\E$ and a bounded function $f$, we have $\gac{f}{\E}{\Phi}=\hi$ regardless of the choice of $\Phi$, so every weak operator integral is a restriction of the bounded operator $\int f d\E$ to some subspace.
\end{example}

\begin{example}\label{trivialexample}
Let $\E(X):=\mu(X)I$, where $\mu$ is a probability measure, and let $f$ be a $\mu$--integrable
function. Then $\formset{f}{\E}=\hi\times \hi$, and $\gac{f}{\E}{\Phi} = \hi$, regardless of the choice of $\Phi$, so that weak operator integrals are simply
restrictions of $\vp\mapsto (\int f \, d\mu)\vp$ to some subspaces of $\hi$. If $f$ is not $\mu$--integrable, then $\formset{f}{\E}=\{ (\psi,\vp)\in \hi\times \hi\mid \langle \psi|\vp\rangle =0\}$, and $\predomna{\vp}$ is the orthogonal complement of $\{ \vp\}$. This is dense only for $\vp=0$, so $\ga{f}{\E}=\{0\}$. Hence, there exists only one weak operator integral, which is the zero operator defined on $\{0\}$.
\end{example}

\subsection{Weak operator integrals determined by a fixed separating subspace}

We now look at the class $\weakops{f}{\E}{\Phi}$ with particular choices of $\Phi$.
The canonical choice would be to take, for each $\vp\in \ga{f}{\E}$, the separating subspace to be the maximal one, i.e.\ $\Phi(\vp)=\predomna{\vp}$ for each $\vp$.
However, in practice, it often happens that a fixed dense subspace (of e.g. smooth functions) is fixed. For example, this can be a linear space spanned by some physically relevant orthonormal basis of $\hi$ (e.g.\ the photon number basis of a single mode optical field).

Accordingly, we now investigate the case where a fixed \emph{dense} subspace $\cal D_s$ is given ($s$ stands for separating). For $\vp\in \ga{f}{\E}$ we define $\Phi_{\cal D_s}(\vp)=\cal D_s$ if $\cal D_s\subseteq \predomna{\vp}$ and $\Phi_{\cal D_s}(\vp)=\predomna{\vp}$ otherwise. Then we have

\begin{proposition} The set
$$
\weakD{f}{\E}{\cal D_s}:=\Big\{ \vp\in \hi\Big| \cal D_s\subseteq \predomna{\vp}, \, \cal D_s\ni \psi\mapsto \int f d\E_{\psi,\vp}\in \C\; \text{\rm is continuous}\Big\}
$$
is the domain of a (clearly unique) element $\weakop{f}{\E}{\cal D_s}\in\weakops{f}{\E}{\Phi_{\cal D_s}}$. In the case where $\E$ is a POVM, this operator is an extension of the maximal strong operator integral $\strgop{f}{\E}$.
\end{proposition}
\begin{proof} Clearly,
$
\weakD{f}{\E}{\cal D_s}=\big\{ \vp\in \gac{f}{\E}{\Phi_{\cal D_s}} \big| \cal D_s\subseteq \predomna{\vp}\big\};
$
in particular, $\weakD{f}{\E}{\cal D_s}$ is a subset of $\gac{f}{\E}{\Phi_{\cal D_s}}$. We have to show that it is a linear space. Let $\vp_1,\vp_2\in \weakD{f}{\E}{\cal D_s}$, and $\alpha ,\beta \in \C$. Now $\cal D_s \subseteq \predomna{\vp_1}\cap\predomna{\vp_2}\subseteq \predomna{\alpha \vp_1+\beta \vp_2}$ (see \eqref{predomincl}); in particular, the latter is dense, so $\alpha \vp_1+\beta \vp_2\in \ga{f}{\E}$, and $\Phi(\alpha \vp_1+\beta \vp_2)=\cal D_s$. Since $\psi\mapsto \int f d\E_{\psi,\vp_i}$ is continuous on $\cal D_s$ for $i=1,2$, then $\psi\mapsto \int f d\E_{\psi,\alpha \vp_1+\beta \vp_2}$ is continuous on $\cal D_s$. Hence, $\alpha \vp_1+\beta \vp_2\in \gac{f}{\E}{\Phi_{\cal D_s}}$. We have shown that $\weakD{f}{\E}{\cal D_s}$ is a linear space. By Proposition \ref{weakconstr}, the restriction of $\preop{f}{\E}{\Phi_{\cal D_s}}$ to $\weakD{f}{\E}{\cal D_s}$ is an element of $\weakops{f}{\E}{\Phi_{\cal D_s}}$. It remains to prove that in the case where $\E$ is a POVM, the domain of the maximal strong operator integral is included in $\weakD{f}{\E}{\cal D_s}$. But this is clear because for $\vp\in \strgdom{f}{\E}$, we have $\cal D_s \subseteq \hi=\predomna{\vp}$, regardless of  $\cal D_s$.
\end{proof}

Since
$\strgop{f}{\E}\subseteq \weakop{f}{\E}{\cal D_s}$ for any POVM $\E$, one can ask when these two operators are the same. Since $\|\eta\|=\sup\{|\langle \psi|\eta\rangle| \mid \psi\in \cal D_s,\, \|\psi\|\leq 1\}$ (as $\cal D_s$ is dense), the following result is a direct consequence of \cite[Theorem 3.5]{Ylinen} (see also \cite{Lewis70}).

\begin{proposition} Suppose $\E$ is a POVM, and let $\cal D_s\subseteq \hi$ be a dense subspace. Then $\strgop{f}{\E}=\weakop{f}{\E}{\cal D_s}$ if and only if for each $\vp\in \weakD{f}{\E}{\cal D_s}$, we have
$$
\lim_{n\rightarrow\infty} \sup_{\psi\in \cal D_s, \,\|\psi\|\leq 1} \int_{X_n} |f| d|\E_{\psi,\vp}|=0
$$
whenever the sets $X_n\in \h A$ satisfy $X_{n+1}\subseteq X_{n}$, $n\in \N$, and $\cap_n X_n=\emptyset$.
\end{proposition}

\subsection{Symmetric weak operator integrals}

Since the integrals $\int f d\E_{\psi,\vp}$ are symmetric in the sense that $(\psi,\vp)\in \formsetna$ implies $(\vp,\psi)\in \formsetna$, and $\overline{\int f d\E_{\psi,\vp}}= \int \overline{f} d\E_{\vp,\psi}$, it is natural to ask when a weak operator integral is a symmetric operator. We will not look at the most general case, but concentrate on the elements of $\weakops{f}{\E}{\Phi_{\cal D_s}}$, with the fixed separating subspace $\cal D_s\subseteq \h V$. Since continuity properties of the integral $\int f d\E_{\psi,\vp}$ with respect to the vectors $\vp,\psi$ are rather weak (even in the case where $\E$ is POVM), knowing that
$$
\langle \psi | \weakop{f}{\E}{\cal D_s}\vp\rangle = \int f d\E_{\psi,\vp}=\overline{\int \overline{f} d\E_{\vp,\psi}}
$$
for all $\psi\in \cal D_s$, $\vp\in \weakD{f}{\E}{\cal D_s}$, is not obviously enough to connect this to the case where $\vp\in \cal D_s$ and $\psi\in \weakD{f}{\E}{\cal D_s}$. Therefore, we now assume that the dense subspace $\cal D_s$ satisfies the equivalent conditions of the following trivial lemma.
\begin{lemma}\label{triv} Let $\cal D_s\subseteq \h V$ be a subspace. Then
$
\cal D_s\subseteq \{\vp\in \hi\mid \cal D_s\subseteq \predomna{\vp}\}
$
if and only if $\cal D_s\times \cal D_s\subseteq \formsetna$.
\end{lemma}
\begin{proposition}\label{weaksymdef} Suppose that $\cal D_s\subseteq \h V$ is a dense subspace satisfying $\cal D_s\times \cal D_s\subseteq \formsetna$. We define a (clearly unique) operator $\weaksymop{f}{\E}{\cal D_s}$ whose domain and action are given by
\begin{align*}
\weaksymD{f}{\E}{\cal D_s}&:=\Big\{\vp \in \cal D_s\Big| \cal D_s\ni\psi\mapsto \int f\, d\E_{\psi,\vp}\in \C\; \text{\rm is continuous}\Big\},\\
\langle \psi|\weaksymop{f}{\E}{\cal D_s}\vp\rangle &=\int f d\E_{\psi,\vp}, \qquad \psi\in \cal D_s,\;\vp\in \weaksymD{f}{\E}{\cal D_s}.
\end{align*}
Then
$\weaksymop{f}{\E}{\cal D_s}\subseteq \weakop{f}{\E}{\cal D_s}.$
In particular, $\weaksymop{f}{\E}{\cal D_s}$ is a weak operator integral, with $\weaksymop{f}{\E}{\cal D_s}\in \weakops{f}{\E}{\Phi_{\cal D_s}}$. Moreover, if $f$ is real-valued, then $\weaksymop{f}{\E}{\cal D_s}$ is a symmetric operator.
\end{proposition}
\begin{proof} It is clear that $\weaksymop{f}{\E}{\cal D_s}$ is a well-defined operator on the given domain $\weaksymD{f}{\E}{\cal D_s}$. (Note that the condition $\cal D_s\times \cal D_s\subseteq \formsetna$ ensures that the integral is defined.) We now show that $\weaksymD{f}{\E}{\cal D_s}\subseteq \weakD{f}{\E}{\cal D_s}$, which by Proposition \ref{weakconstr} implies that $\weaksymop{f}{\E}{\cal D_s}\in \weakops{f}{\E}{\Phi_{\cal D_s}}$ and $\weaksymop{f}{\E}{\cal D_s}\subseteq \weakop{f}{\E}{\cal D_s}$. Accordingly, let $\vp\in \weaksymD{f}{\E}{\cal D_s}$. In particular, $\vp\in \cal D_s$. Since $\cal D_s\times \cal D_s\subseteq \formsetna$, we have $\cal D_s\subseteq \predomna{\vp}$. But $\psi\mapsto \int f \, d\E_{\psi,\vp}$ is continuous on $\cal D_s$, so $\vp\in \weakD{f}{\E}{\cal D_s}$. It remains to show that $\weaksymop{f}{\E}{\cal D_s}$ is symmetric if $f$ is real-valued. For that, let $\psi,\vp\in \weaksymD{f}{\E}{\cal D_s}$. Then both of them are also in $\cal D_s$. Hence,
\begin{align*}
\langle \psi |\weaksymop{f}{\E}{\cal D_s}\vp\rangle = \int f\,d\E_{\psi,\vp}=\overline{\int f\, d\E_{\vp,\psi}}
=\overline{\langle \vp |\weaksymop{f}{\E}{\cal D_s}\psi\rangle}=\langle \weaksymop{f}{\E}{\cal D_s}\psi|\vp\rangle.
\end{align*}
\end{proof}

We call an operator $\weaksymop{f}{\E}{\cal D_s}$ \emph{symmetric weak operator integral} determined by $\cal D_s$. (Even in the case where $f$ is not real valued.)

\begin{example}
Suppose that $\hi$ is separable, let $\{\vp_n\}_{n\in \N}$ be an orthonormal basis of $\hi$, and put $\vi:={\rm lin}\{\vp_n\,|\,n\in \N\}$,
and $\E:\,\cal A\to\sv$ a positive sesquilinear form measure.
Let $f:\,\cal A\to\C$ be such that
\begin{equation}\label{continuity}
\sum_{n\in \N}\big|\int f d\E_{\vp_n,\vp_m}\big|^2<\infty \text{ for all } m\in \N.
\end{equation}
In particular, $\int f d\E_{\vp_n,\vp_m}$ exists for all $n,m\in \N$, that is, $(\vp_n,\vp_m)\in \formsetna$ for each $n$. By sesquilinearity, it follows that $(\psi,\vp)\in \formsetna$ for all $\psi,\vp\in \vi$, i.e. $\formsetna= \vi \times \vi$. Hence, $\vi$ itself satisfies the conditions of Proposition \ref{weaksymdef}, and we have the symmetric weak operator integral $\weaksymop{f}{\E}{\vi}$.
It now follows from \eqref{continuity} that for each $m\in \N$,
$$
\vi\ni\psi\mapsto\int f d\E_{\psi,\vp_m}=\sum_{n\in N}\<\psi|\vp_n\>
\int f d\E_{\vp_n,\vp_m}
\in\C
$$
is continuous. Since each $\vp\in \vi$ is a (finite) linear combination of the vectors $\vp_m$, the continuity holds for each $\vp\in \vi$. Hence, the domain of the symmetric weak operator integral $\weaksymop{f}{\E}{\vi}$ is the whole of $\vi$, and its action is determined by
$$
\weaksymop{f}{\E}{\vi}\vp_m=\sum_{n\in N} \left(\int f d\E_{\vp_n,\vp_m}\right) \vp_n, \text{ for all } m\in \N.
$$
Of course, an operator defined via this same formula may have a larger domain; for example, if
$$
\E_{\vp_n,\vp_m}(X)=\delta_{nm}\mu_n(X),\qquad n,m\in N,\;
X\in\cal A,
$$
where $\delta_{nm}$ is the Kronecker delta and
 $\{\mu_n\}$ is a sequence\footnote{Obviously, $\E$ defines a POVM if and only if
$\sup_{n\in\N} \mu_n(\Omega)<\infty.$}
 of bounded positive measures on $\cal A\subseteq 2^\Omega$ then $\int f d\E_{\vp_n,\vp_m}=\delta_{nm} f_m$, where $f_m:=\int f d\mu_m$, and
the largest possible domain of an extension of the weak operator integral $\weaksymop{f}{\E}{\vi}$ is
$\{\vp\in\hi | \sum_m |f_m\<\vp_m|\vp\>|^2<\infty\}$. Note that this extension is bounded if and only if $\sup_{m\in\N}|f_m|<\infty$. However, the extension is \emph{not} a weak operator integral, because its domain is larger than the form domain $\vi$ of the sesquilinear form valued measure $\E$.
\end{example}

We immediately notice the following
\begin{proposition}\label{strgweakprop} Suppose $\E$ is a POVM, and the strong operator integral $\strgop{f}{\E}$ is densely defined. Set $\cal D_s=\strgdom{f}{\E}$. Then $\strgop{f}{\E}=\weaksymop{f}{\E}{\cal D_s}$, i.e.\ the strong operator integral is the symmetric weak operator integral determined by its domain.
\end{proposition}
\begin{proof} If $\vp\in \cal D_s$ then $\int f d\E_{\psi,\vp}$ exists for all $\psi\in \hi$, so $\cal D_s\times\cal D_s\subseteq \hi\times \cal D_s\subseteq \formsetna$. Hence, $\weaksymop{f}{\E}{\cal D_s}$ is defined. Moreover, if $\vp \in \cal D_s$ then $\psi\mapsto \int fd\E_{\psi,\vp}$ is continuous on the whole $\hi$, and hence also on the subspace $\cal D_s$. Thus $\cal D_s\subseteq \weaksymD{f}{\E}{\cal D_s}\subseteq \cal D_s$, and the proof is complete.
\end{proof}

Hence, the domain of the strong operator integral, when dense, is one choice for a separating subspace $\cal D_s$ of a weak operator integral when $\E$ is a POVM. It it easy to see that even in the general case there is a \emph{maximal choice} for this subspace, which can be explicitly written down:
\begin{proposition}\label{formdomdef} The set
$$
\formdom{f}{\E}:=\Big\{\vp\in\hi\Big| \int |f| d\E_{\vp,\vp}<\infty\Big\}=\{\vp\in\hi\mid \vp\in \predomna{\vp}\}
$$
is the largest subspace $\cal D\subseteq\hi$ such that
$\cal D\times\cal D\subseteq \formset{f}{\E}$ (in the sense that any subspace $\cal D$ with this property, is included in $\formdom{f}{\E}$).
\end{proposition}
\begin{proof}
The fact that the set $\formdom{f}{\E}$ is a linear subspace of $\hi$ follows immediately from Lemma \ref{formineq}. Next we note that given $\vp,\psi\in \hi$, the measure $\E_{\psi,\vp}$ is a linear combination of four measures of the form $\E_{\psi+ i^k\vp,\psi+ i^k\vp}$, $k=0,1,2,3$. If $(\psi,\psi)\in \formset{f}{\E}$ and $(\vp,\vp)\in \formset{f}{\E}$ then $f$ is integrable with respect to each of the four measures, since $\formdom{f}{\E}$ is a linear subspace. Hence, $f$ is also integrable with respect to $\E_{\psi,\vp}$, that is, $(\psi,\vp)\in \formset{f}{\E}$. Thus, $\formdom{f}{\E}\times\formdom{f}{\E}\subseteq\formset{f}{\E}$. On the other hand, if $\cal D\subseteq \hi$ is any subspace with $\cal D\times\cal D\subseteq\formset{f}{\E}$, then $(\vp,\vp)\in \formset{f}{\E}$ for all $\vp\in \cal D$, so that $\cal D\subseteq \formdom{f}{\E}$. Thus, $\formdom{f}{\E}$ is the largest of such subspaces $\cal D$.
\end{proof}

Assuming that $\formdomna$ is dense, we denote
$$\mweaksymop{f}{\E}:=\weaksymop{f}{\E}{\formdomna},$$
and call this \emph{the largest symmetric weak operator integral} determined by $f$ and $\E$. All other symmetric operator integrals are restrictions of this one. In particular, if $\E$ is a POVM, Proposition \ref{strgweakprop} gives
$$
\strgop{f}{\E}\subseteq \mweaksymop{f}{\E}.
$$
Note that this inclusion holds even in the case where $\strgop{f}{\E}$ is not dense (which can easily happen even if $\formdomna$ is dense), because if $\int f d\E_{\psi,\vp}$ exists for all $\psi$ then $\int |f| d\E_{\vp,\vp}<\infty$.

The following result deals with the case of spectral measures.
\begin{proposition}\label{spectral}
Suppose that $\E$ is projection valued. Then
$$\sqdop{f}{\E}=\strgop{f}{\E}=\mweaksymop{f}{\E}.$$
\end{proposition}
\begin{proof}
Since $\sqdop{f}{\E}=\strgop{f}{\E}$ is densely defined (the usual spectral integral), the weak operator integral $\mweaksymop{f}{\E}$ exists, and is an extension of $\strgop{f}{\E}$. Hence, we only need to show that $\dom(\mweaksymop{f}{\E})\subseteq \sqd{f}{\E}$. Define $g:\Omega\to \C$ by $g = \sqrt{|f|}$, and $h:\Omega \to \C$ by setting $h( x) = f( x)/(|f( x)|)$ if $f( x)\neq 0$, and $h( x)=0$ otherwise. Then $h$ and $g$ are measurable, $h$ is bounded, $g\geq 0$, and $f=g^2h$. Now
\begin{equation}\label{compos}
\strgop{g}{\E}^*\strgop{gh}{\E}\subseteq \strgop{g^2h}{\E} = \strgop{f}{\E},
\end{equation}
by the usual rules of spectral calculus of unbounded functions. Now
$$
\formdomna=\{ \vp\in \hi \mid \int |f| \, d\E_{\vp,\vp} <\infty \} = \dom(\strgop{g}{\E}) = \dom(\strgop{gh}{\E}).
$$
According to what has been concluded earlier by using polarization, $f$ is $\E_{\psi,\vp}$-integrable whenever
\emph{both} $\psi$ and $\vp$ belong to $\formdomna$.
Since $\E$ is a spectral measure, we have
$$
\int f d\E_{\psi,\vp} = \int g(gh)\, d\E_{\psi,\vp}=\langle \strgop{g}{\E}\psi |\strgop{gh}{\E}\vp\rangle, \ \vp,\psi\in \formdomna.
$$
Indeed, if $g$ is bounded, then this follows from the multiplicativity of the spectral measure, and in the general case, we approximate $g$ with the sequence $(g_n)$, where $g_n( x) = g( x)$ if $g( x) \leq n$, and $g_n( x) =0$ otherwise, and conclude that on the one hand, $\strgop{g_nh}{\E}\vp\rightarrow \strgop{gh}{\E}$, $\strgop{g_n}{\E}\vp\rightarrow \strgop{g}{\E}$ strongly,
and on the other hand, $\int g_n^2h\, d\E_{\psi,\vp}\rightarrow \int g^2h \, d\E_{\psi,\vp}$ by dominated convergence (since $|f|=g^2$ is $\E_{\psi,\vp}$-integrable).

Now if $\vp\in \dom(\mweaksymop{f}{\E})$ then by definition, $\psi\mapsto \int f\, d\E_{\psi,\vp}$ is continuous in $\formdomna=\dom(\strgop{g}{\E})$. By the formula obtained, this implies that $\strgop{gh}{\E}\vp$ belongs to
$\dom(\strgop{g}{\E}^*)$, i.e.\ $\vp\in \dom(\strgop{g}{\E}^*\strgop{gh}{\E}$, so $\vp\in \dom(\strgop{f}{\E})$. The proof is complete.
\end{proof}

\section{Sesquilinear form valued integral}

\subsection{The sesquilinear form valued integral of a sesquilinear form valued measure and a measurable function}

Since $\E$ is a sesquilinear form valued measure, it is natural to consider the \emph{sesquilinear form valued integral} of a measurable function with respect to $\E$. In this section, we first define this integral, and then consider its connection to weak operator integrals.

We start by defining a function
\begin{align*}
\form{f}{\E}:\,\formset{f}{\E}\to \C, \qquad (\psi,\varphi)\mapsto\form{f}{\E}(\psi,\varphi)=\int f d\E_{\psi,\varphi}.
\end{align*}
This function satisfies e.g.
$(\alpha \psi_1+\beta \psi_2,\varphi)\in \formset{f}{\E}$ and $\form{f}{\E}(\alpha \psi_1+\beta \psi_2,\varphi)=\overline{\alpha }\form{f}{\E}(\psi_1,\varphi)+\overline{\beta }\form{f}{\E}(\psi_2,\varphi)$,
for any $(\psi_1,\varphi),(\psi_2,\varphi)\in \formset{f}{\E}$.
In addition, $(\psi,\vp)\in \formset{f}{\E}$ if and only if $(\vp,\psi)\in \formset{f}{\E}$, and
\begin{equation}\label{symmetry}
\overline{\form{f}{\E}(\psi,\vp)} = \form{\overline f}{\E}(\vp,\psi), \qquad (\psi,\vp)\in \formset{f}{\E}.
\end{equation}
In order to consider $\form{f}{\E}$ as a sesquilinear form, we have to restrict its domain of definition to a set of the form $\cal D\times \cal D\subseteq \formset{f}{\E}$, where $\cal D\subseteq \vi$ is a subspace. (Clearly, any such restriction is sesquilinear.) According to Proposition \ref{formdomdef}, there is a canonical choice for $\cal D$, namely the largest one $\formdomna$. We denote the restriction of $\formna$ to $\formdomna$ by the same symbol. We say that
$$
\formna:\formdomna\times\formdomna\to\C
$$
is \emph{the form integral of $f$ with respect to $\E$}. The subspace $\formdomna$ is the \emph{form domain}.

It follows from \eqref{symmetry} that $\formna$ is symmetric if $f$ is real valued. It is clearly positive if $f$ is a positive function.

\begin{remark} The form domain should not be confused with the square integrability domain, which is the form domain of the form integral of $|f|^2$ with respect to $\E$. In the case of $f(x)=x$ on $\R$, the latter is called \emph{variance form;} see Introduction.
\end{remark}

In order to consider the connection between the (unique) form integral of $f$ with respect to $\E$, and the various weak operator integrals, we need some preliminaries on the standard extension theory of quadratic forms.

\subsection{Preliminaries on quadratic forms}

We start with some basic preliminaries on the theory of quadratic forms (see e.g.\ \cite{RS1,Kato}). \emph{A quadratic form} is a sesquilinear form $q:\cal D\times \cal D\to \C$, where $\cal D\subseteq \hi$ is a {\it dense} subspace, called the \emph{form domain}. If $q(\psi,\vp)=\overline{q(\vp,\psi)}$, for all $\psi,\vp\in \cal D$, then $q$ is called \emph{symmetric}, and if $q(\vp,\vp)\geq 0$ for all $\vp\in \cal D$, it is called \emph{positive}.

The \emph{adjoint form} $q^*$ of $q$ is defined on the same domain $\cal D$, via
$$q^*(\vp,\psi)=\overline{q(\psi,\vp)}, \qquad \vp,\,\psi\in \cal D.$$ Inclusion $q'\subseteq q$ between two quadratic forms is defined via the corresponding inclusion of the form domains. A linear combination of two quadratic forms is defined in the obvious way, with the domain being the intersection of the form domains. In particular, the \emph{real} and \emph{imaginary parts} of a quadratic form $q$ are defined by
\begin{align*}
\Re q&:=\frac 12 (q+q^*), & \Im q&:=\frac {1}{2i} (q-q^*).
\end{align*}
A positive quadratic form $q:\cal D\times \cal D\to \C$ is said to be \emph{closed} if $\vp_n\in \cal D$, $\vp_n\rightarrow \vp\in \hi$, and $$\lim_{n,m\rightarrow \infty} q(\vp_n-\vp_m,\vp_n-\vp_m)=0$$
imply $\vp\in \cal D$ and
$$
\lim_{n\rightarrow\infty} q(\vp_n-\vp,\vp_n-\vp)=0.
$$
It follows that $q$ is closed if and only if $\Re q$ is closed (see \cite[p.\ 313]{Kato}.

There is a canonical way of associating a positive selfadjoint operator to a positive closed quadratic form. It is given by the following theorem (see \cite{Kato,EBD1}).
\begin{theorem}\label{Katorep} Let $q$ be a closed symmetric positive quadratic form with dense form domain $\cal D$. Then there exists a positive selfadjoint operator $T$ such that $\dom(\sqrt{T})=\cal D$, and
$$
q(\psi,\vp)=\langle \sqrt{T}\psi|\sqrt{T}\vp\rangle, \quad \text{ for all }\psi,\vp\in \cal D.
$$
\end{theorem}

We say that $T$ given by the above theorem is the \emph{operator associated to the quadratic form} $q$. 
We will make use of the following simple corollary; it also shows that $T$ is uniquely determined, hence the definite article.

\begin{proposition}\label{formop} Let $q$ be a closed symmetric positive quadratic form with dense form domain $\cal D\subseteq \hi$ and $T$ a positive positive selfadjoint operator associated to it as in Theorem \ref{Katorep}. Then
\begin{align}\label{formopdef}
\dom(T)&=\{\vp\in \cal D\mid \cal D\ni \psi\mapsto q(\psi,\vp)\in \C \; \text{\rm is continuous }\},\nonumber\\
q(\psi,\vp)&=\langle \psi|T\vp\rangle, \quad \text{ for all }\psi\in \cal D, \vp\in \dom(T).
\end{align}
If there is a Hilbert space $\cal K$, and an operator $A:\cal D\to \cal K$, such that
$$
q(\psi,\vp)=\langle A\psi|A\vp\rangle, \quad \text{ for all }\psi,\vp\in \cal D,
$$
then $A^*A=T$.
\end{proposition}
\begin{proof} If $A$ is as in the lemma, we have, by the definition of the adjoint, that
\begin{align*}
\dom(A^*A) &= \{\vp\in \cal D\mid A\vp\in \dom(A^*)\}\\
&=\{\vp\in \cal D\mid \cal D\ni \psi\mapsto \langle A\psi|A\vp\rangle\in \C \text{ is continuous }\}\\
&=\{\vp\in \cal D\mid \cal D\ni \psi\mapsto q(\psi,\vp)\in \C \text{ is continuous }\}.
\end{align*}
In particular, this holds for $\cal K=\hi$ and $A=\sqrt{T}$, which gives \eqref{formopdef}, because $T=(\sqrt{T})^*\sqrt{T}$. It follows that $\dom(A^*A)=\dom(T)$, and if $\vp\in \dom(T)$, $\psi\in \cal D$, we have
$$\langle \psi|A^*A\vp\rangle = \langle A\psi|A\vp\rangle =q(\psi,\vp)=\langle \psi|T\vp\rangle.$$
As $\cal D$ is dense, this implies that $A^*A=T$.
\end{proof}
\begin{remark} Note that it is nontrivial that the domain of $A^*A$ in the above proposition is actually dense. This fact follows from the above theorem. Note also that the operator $A$ is automatically closed, because the form $q$ was assumed to be closed. A special case of this result is the well-known theorem of von Neumann (see e.g.\ \cite[p.\ 180]{RS2}), which says that $A^*A$ is selfadjoint if $A$ is a closed densely defined operator.
\end{remark}

\subsection{Connection between form integral and weak operator integral}

In the case where $\formdomna$ is dense, the sesquilinear form $\formna$ is a quadratic form. In general, the adjoint form is given by
$$
\formna^*(\psi,\vp)= \int \overline{f} d\E_{\psi,\vp}, \, \psi,\vp\in \formdomna.
$$
Moreover,
\begin{align*}
\Re(\formna)&\subseteq \form{\Re(f)}{\E},& \Im(\formna)&\subseteq \form{\Im(f)}{\E},
\end{align*}
where $\Re(f)$ and $\Im(f)$ are the real and imaginary parts of the function $f$, respectively. We can further decompose these into positive and negative parts, so that
\begin{align*}
f&=f_1-f_2+i(f_3-f_4),\\
&f_i\geq 0, & |\Re(f)|&=f_1+f_2, & |\Im(f)|&=f_3+f_4.
\end{align*}
Then clearly $\formdom{f}{\E}=\cap_{i=1}^4\formdom{f_i}{\E}$, so the form integral decomposes naturally as the linear combination of the corresponding positive forms:
\begin{align*}
\form{f}{\E}&=\form{f_1}{\E}-\form{f_2}{\E}+i\form{f_3}{\E}-i\form{f_4}{\E}.
\end{align*}
Unfortunately, the situation is not so simple in case of the weak operator integrals. However, the following result holds:
\begin{proposition} Suppose that $\formdomna$ is dense. Then 
\begin{align*}
\mweaksymop{f}{\E}\supseteq \mweaksymop{f_1}{\E}-\mweaksymop{f_2}{\E}+i\mweaksymop{f_3}{\E}-i\mweaksymop{f_4}{\E}\in \weakops{f}{\E}{\Phi_{\mathcal D_s}},
\end{align*}
where $\mathcal D_s=\formdomna$, and the inclusion can be interpreted as the ordering relation in the class $\weakops{f}{\E}{\Phi_{\mathcal D_s}}$ of weak operator integrals.
\end{proposition}
\begin{proof} First note that since $\formdomna$ is dense, so is each $\formdom{f_i}{\E}$; hence, the weak operator integrals $\mweaksymop{f_i}{\E}$ are defined. Denote $A:=\mweaksymop{f_1}{\E}-\mweaksymop{f_2}{\E}+i\mweaksymop{f_3}{\E}-i\mweaksymop{f_4}{\E}$. By definition, $$\dom(A)=\cap_{i=1}^4\dom(\mweaksymop{f_i}{\E}),$$ so that
$\dom(A)\subseteq \formdomna$. If $\varphi\in \dom(A)$, each functional $\psi\mapsto \int f_i d\E_{\psi,\varphi}$ is continuous on $\formdom{f}{\E}=\cap_{i=1}^4\formdom{f_i}{\E}$, and coincides with $\psi\mapsto \langle \psi|\mweaksymop{f_i}{\E}\varphi\rangle$ there. This implies that $\varphi\in \dom(\mweaksymop{f}{\E})$. Hence, $A\subseteq \mweaksymop{f}{\E}$. Since $\mweaksymop{f}{\E}\in \weakops{f}{\E}{\Phi_{\mathcal D_s}}$, it follows from Proposition \ref{weakconstr} that also $A\in \weakops{f}{\E}{\Phi_{\mathcal D_s}}$. This completes the proof.
\end{proof}

We now consider the relationship between $\form{f}{\E}$ and $\mweaksymop{f}{\E}$ in the case of a positive function $f$, and a POVM $\E$.
\begin{proposition} \label{propotatsioone10}
Let $\E$ be a POVM and $f:\Omega\to \C$ a positive measurable function, such that $\formdomna$ is dense. Then the quadratic form $\formna$ is symmetric, positive and closed. The associated positive selfadjoint operator $T$ (see Theorem \ref{Katorep}), is given by
$$T=(\strgop{\sqrt{f}}{\F}V)^*\strgop{\sqrt{f}}{\F}V,$$
where $(\cal K,\F,V)$ is any Naimark dilation of $\E$. Moreover,
$$
T=\mweaksymop{f}{\E},
$$
i.e., $T$ is the largest symmetric weak operator integral determined by $f$ and $\E$. In particular, $\mweaksymop{f}{\E}$ is selfadjoint.
\end{proposition}
\begin{proof} Clearly, $\formna(\vp,\vp)\geq 0$ for all $\vp\in \formna$. Let $(\cal K, \F, V)$ be a Naimark dilation of $\E$, so that $\E(X)= V^*\F(X)V$ and $\F$ is projection valued. Now
\begin{align}\label{eqform}
\formdomna &= \{\vp\in \hi \mid V\vp\in \strgdom{\sqrt{f}}{\F} \},\nonumber\\
\formna(\vp,\vp) &=\|\strgop{\sqrt{f}}{\F}V\vp\|^2, \, \text{ for all }\vp\in \formdomna.
\end{align}
Now if $\vp_n\in \formdomna$, such that $\vp_n\rightarrow \vp\in \hi$, and $\lim_{n,m\rightarrow \infty} \formna(\vp_n-\vp_m,\vp_n-\vp_m)=0$, it follows that $V\vp_n\rightarrow V\vp$, and $(\strgop{\sqrt{f}}{\F}V\vp_n)_n$ converges in $\cal K$. Since $\F$ is projection valued, $\strgop{\sqrt{f}}{\F}$ is a closed operator on its domain, so $V\vp\in \strgdom{\sqrt{f}}{\F}$, and $\lim_{n\to\infty} \strgop{\sqrt{f}}{\F}V\vp_n=\strgop{\sqrt{f}}{\F}V\vp$. But this implies that $\vp\in \formdomna$, and $\lim_{n\rightarrow\infty} \formna(\vp_n-\vp,\vp_n-\vp)=0$. Hence the form $\formna$ is closed. From \eqref{eqform} it now follows by polarization and Proposition \ref{formop} that $(\strgop{\sqrt{f}}{\F}V)^*\strgop{\sqrt{f}}{\F}V$ is the selfadjoint operator associated to the form $\formna$. From Proposition \ref{formop}, we immediately see that $T=\mweaksymop{f}{\E}$. This completes the proof.
\end{proof}

\begin{remark}\rm Notice that in the above proposition, $V^*\strgop{f}{\F}V\subseteq T=\mweaksymop{f}{\E},$ because
\begin{equation}\label{inclusion}
V^*\strgop{\sqrt{f}}{\F}\subseteq (\strgop{\sqrt{f}}{\F}V)^*.
\end{equation}
From \eqref{squarenaimark}, we know that
$V^*\strgop{f}{\F}V=\sqdop{f}{\E}$. Hence, in this case, the difference between the strong operator integral on the square integrability domain and the maximal symmetric weak operator integral, is in the operator inclusion \eqref{inclusion}, which can be proper because continuity of the functional $\psi\mapsto \langle \strgop{\sqrt{f}}{\F}\psi |\vp\rangle$ on $V(\hi)\cap \dom(\strgop{\sqrt{f}}{\F})$ does not necessarily imply its continuity on the full domain $\dom(\strgop{\sqrt{f}}{\F})$.
\end{remark}

\section{Application: moment operators of a POVM}\label{viimeinenluku}

Consider the \emph{operator valued moments} (or \emph{moment operators}) of a normalized POVM $\E:\cal B(\R)\to \lh$. They
are simply defined as operator integrals $\int x^k d\E$ of real functions $x\mapsto x^k$ where $k\in \N$.\footnote{For simplicity, we will use the symbol $x^k$ to denote the function $x\mapsto x^k$.}
Hence, we have three natural ways
to defined them: $\tilde{\E}[k]:=\tilde{L}(x^k,\E)$, $\E[k]:=L(x^k,\E)$, and ${\E}'[k]:={L}'(x^k,\E)$.
Recall that $\tilde{\E}[k]\subseteq \E[k]\subseteq {\E}'[k]$ and their domains are
\begin{eqnarray*}
\dom(\tilde{\E}[k]) &=& \Big\{ \vp\in \hi \,\Big| \int x^{2k}\, d\E_{\vp,\vp}(x) <\infty\Big\}, \\
\dom({\E}[k]) &=& \Big\{ \varphi\in \hi \, \Big| \int |x|^k d|\E_{\psi,\varphi}|<\infty \text{ for all }\psi\in \hi\Big\}, \\
\dom({\E'}[k])&=&\Big\{\vp \in \cal D_F(x^k,\E)\Big| \cal D_F(x^k,\E)\ni\psi\mapsto \int x^k\, d\E_{\psi,\vp}\in \C\; \text{\rm is continuous}\Big\}\\
\text{if }&& \formdom{x^k}{\E}=\Big\{\vp\in\hi\Big| \int |x|^k d\E_{\vp,\vp}<\infty\Big\}\;\text{ is dense in }\hi.
\end{eqnarray*}
By comparison, the \emph{sesquilinear form valued moments} are given by the form integral
\begin{align*}
\form{x^k}{\E}(\psi,\varphi) &=\int x^k d\E_{\psi,\varphi}, & \psi,\varphi\in \formdom{x^k}{\E}.
\end{align*}
Since $\dom(\tilde{\E}[k])= \formdom{x^{2k}}{\E}$, we can define the \emph{variance form} on this form domain as
$$
(\psi,\varphi)\ni \form{x^{2k}}{\E}(\psi,\varphi)-\langle \tilde{\E}[k]\psi|\tilde{\E}[k]\varphi\rangle.
$$
If this form is identically zero, the POVM is called \emph{variance-free} \cite{Werner}.

Let $(\cal K,\F,V)$ be a Naimark dilation of $\E$.
From Proposition \ref{spectral} one sees that
$\tilde{\F}[k]=\F[k]=\F'[k]$ for any $k\in\N$. Moreover,
$\dom(\F[k]V)=\dom(\tilde{\E}[k])$ 
and
\begin{equation}\label{tildeoperators}
\tilde{\E}[k]=V^*\F[k]V.
\end{equation}
(This was also proved in \cite{LahtiII}.)
If $k$ is even, i.e.\ $k=2j$, $j\in\N$, and
$
\formdom{x^{2j}}{\E}=\dom(\tilde{\E}[j])
$
is dense {\it (which is assumed below)}, then it follows from Proposition \ref{propotatsioone10}
that
\begin{equation}\label{pilkkuoperators}
{\E}'[2j]=(\F[j]V)^*(\F[j]V)
\end{equation}
is positive and selfadjoint. Now
$\dom({\E}'[2j])$ consists of exactly those vectors $\vp\in \dom(\tilde{\E}[j])$ for which
$\F[j]V\vp\in\dom\big((\F[j]V)^*\big)$.
Note that $\F[j]V$ is a map $\dom(\tilde{\E}[j])\to \cal K$,
so the adjoint $(\F[j]V)^*$ maps from a subspace of $\cal K$ to $\cal H$.
It is clear that
\begin{equation}\label{incluusio}
V^*\F[j]\subseteq (\F[j]V)^*
\end{equation}
because if $\vp\in \dom(\F[j])$, then $\psi\mapsto \langle \vp| \F[j]V\psi\rangle= \langle V^*\F[j]\vp|\psi\rangle$ is continuous in
$\dom(\F[j]V)$.

\begin{proposition}\label{trivialprop}
\begin{itemize}
\item[(a)] Suppose that $V^*\F[j]= (\F[j]V)^*$. Then $\tilde{\E}[2j]={\E}'[2j]$.
\item[(b)] Suppose that $\F[j]\big(\dom(\F[j])\cap V(\hi)\big)\subseteq V(\hi)$. Then
${\E}'[2j] = \tilde{\E}[j]^*\tilde{\E}[j]$.
\end{itemize}
\end{proposition}
\begin{proof}
We have already proved (a); see \eqref{tildeoperators}, \eqref{pilkkuoperators}, and use $\F[2j]=\F[j]\F[j]$.
 To prove (b), note that the assumption implies $\F[j]V=VV^*\F[j]V=V\tilde{\E}[j]$.
Now a vector $\vp\in \cal H$ satisfies $V\vp\in \dom\big((V\tilde{\E}[j])^*\big)$ if and only if $\psi\mapsto \langle V\tilde{\E}[j]\psi |V\vp\rangle = \langle \tilde{\E}[j]\psi|\vp\rangle$ is continuous on $\dom(\tilde{\E}[j])$, which happens exactly when $\vp\in \dom(\tilde{\E}[j]^*)$. Hence, $(\F[j]V)^*V=(V\tilde{\E}[j])^*V = \tilde{\E}[j]^*$ and
${\E}'[2j] = (\F[j]V)^*(\F[j]V)= (V\tilde{\E}[j])^*(V\tilde{\E}[j]) = \big\{(V\tilde{\E}[j])^*V\big\}\tilde{\E}[j]=\tilde{\E}[j]^*\tilde{\E}[j].$
\end{proof}

\subsection{Momentum for a  bounded interval}
Consider first a free (nonrelativistic) particle of mass $m$ moving along a line which can be chosen to be $\R$ without restricting generality. We use units where $\hbar=1$.
Then the Hilbert space of the system is $L^2(\R)$ and the (sharp) position observable is $\Qo_\R:\,\cal B(\R)\to\LH$,
$$
(\Qo_\R(X)\psi)(x):=\CHI X(x)\psi(x),\qquad X\in\bo\R,\;\psi\in L^2(\R),\;x\in\R.
$$
The (sharp) momentum observable is
$\Po_\R:\,\cal B(\R)\to\LH$
$$
\Po_\R(Y):=\cal F^*\Qo_\R(Y) \cal F,\qquad Y\in\bo\R,
$$
where
$\mathcal{F}:L^2(\mathbb{R})\to L^2(\mathbb{R})$ is the
Fourier-Plancherel (unitary) operator determined by
$$
(\mathcal{F}\psi)(x) = \frac{1}{\sqrt{2\pi}} \int_\R
e^{-ixt} \psi(t)\, dt, \qquad \psi\in
L^1(\R)\cap L^2(\R),\;x\in \mathbb{R}.
$$
Since $\Qo_\R$ and $\Po_\R$ are spectral measures, there is no ambiguity in defining their moment operators, see Proposition \ref{spectral}. For example,
$\Qo_\R[1]$ and $\Po_\R[1]$ are the usual selfadjoint position and momentum operators,
$$
(\Qo_\R[1]\psi)(x)=x\psi(x),\qquad
(\Po_\R[1]\psi)(x)=-i \psi'(x)
$$
where $\psi'(x):=d\psi(x)/dx$ (and similarly $\psi''(x):=d^2\psi(x)/dx^2$).
Recall that, e.g., $\dom(\Po_\R[1])$
consists of those absolutely continuous functions $\psi\in L^2(\R)$ for which $\psi'\in L^2(\R)$.
Now the energy operator is $(2m)^{-1}\Po_\R[1]^2=(2m)^{-1}\Po_\R[2]$ whose spectrum is continuous, consisting of nonnegative numbers.

Suppose then that the particle is confined to move on a (fixed) bounded interval taken to be $\cal I=[0,\ell]$ where $\ell>0$ is the length of the interval. Note that we do not assume that the endpoints $0$ and $\ell$ can be identified so that the system is not periodic with periodic boundary conditions (indeed, in the periodic case, the position space is a circle instead of an interval).

Since the particle is strictly confined to the interval $\cal I$, the Hilbert space of the system is $L^2(\cal I)$ and the position observable is now the (restricted) spectral measure $\Qo:\,\cal B(\R)\to\LHH$,
$$
(\Qo(X)\vp)(x):=\CHI X(x)\vp(x),\qquad X\in\bo\R,\;\vp\in L^2(\cal I),\;x\in\cal I.
$$
Indeed, let $U:\,L^2(\cal I)\to L^2(\R)$ be the
isometry $(U\vp)(x) = \vp(x)$ for $x\in {\cal I}$, and $(U\vp)(x)=0$, $x\notin {\cal I}$.
Then $U^*:\,L^2(\R)\to L^2(\cal I)$ simply acts as $(U^*\psi)(x) = \psi(x)$, $x\in {\cal I}$, and
$$
\Qo(X)=U^*\Qo_\R(X)U,\qquad X\in\bo\R.
$$
Again, there is no ambiguity in calculating the moments of $\Qo$.
However, the situation is totally different for the momentum POVM $\Po:\,\bo\R\to\LHH$ which is defined similarly to $\Qo$:
$$
\Po(Y):=U^*\Po_\R(Y)U,\qquad Y\in\bo\R.
$$
Note that $(L^2(\R),\Po_\R,U)$ is a Naimark dilation of $\Po$.

The following questions now arise:
{\it What is the correct definition for the second moment operator of $\Po$?
Is the second moment of $\Po$ (times $(2m)^{-1}$) the energy operator in this case?}

The operators $\Po_\R[1]U$, $(\Po_\R[1]U)^*$, and $\tilde{\Po}[1]$ can now be explicitly determined, but a certain care has to be excercised. Namely, the domain of $\Po_\R[1]U$ consists of exactly those functions $\vp\in L^2(\cal I)$ for which $U\vp$ is absolutely continuous, with $(U\vp)'\in L^2(\R)$. Now $U\vp$ is absolutely continuous exactly when $\vp$ is absolutely continuous
in the interval $\cal I=[0,\ell]$, \emph{and vanishes at the endpoints}. (If it did not vanish, then there would be a discontinuity.) The set of absolutely continuous functions $\vp\in L^2(\cal I)$ with $\vp'\in L^2(\cal I)$
and $\vp(0)=\vp(\ell)=0$ is denoted by $\dom(P_0)$, and the corresponding version of the differential operator $-i{d}/{dx}$ by $P_0$, acting in $L^2(\cal I)$.
Hence, $\Po_\R[1]U=UP_0$. This implies $\tilde{\Po}[1] = U^*\Po_\R[1]U = U^*UP_0=P_0$, see \eqref{tildeoperators}. The operator $P_0$ is well known to be densely defined and closed (see e.g. \cite{RS1}).

Now the adjoint of $\Po_\R[1]U=UP_0$ is a map from a subspace of $L^2(\R)$ to $L^2(\cal I)$. A vector $\psi\in L^2(\R)$ belongs to its domain exactly when
$
\vp\mapsto \langle \psi | UP_0\vp\rangle = \langle U^*\psi | P_0\vp\rangle
$
is continuous in $\dom(P_0)\subseteq L^2(\cal I)$. But this happens exactly when $U^*\psi=\psi|_{[0,\ell]}\in \dom(P_0^*)$.
Now $\dom(P_0^*)$ consists of those vectors $\vp\in L^2(\cal I)$ which are absolutely continuous, with $\vp'\in L^2(\cal I)$, and no other restriction. Hence,
$$
\dom\big((\Po_\R[1]U)^*\big) = \big\{ \psi\in L^2(\R)\,\big| \, \psi|_{[0,\ell]} \text{ is absolutely continuous and } {\psi|_{[0,\ell]}}'\in L^2(\cal I) \big\}.
$$
Obviously, this contains $\dom(\Po_\R[1])$, as required by the general inclusion $U^*\Po_\R[1]\subseteq (\Po_\R[1]U)^*$, see \eqref{incluusio}.
Now it is clear that $U^*\Po_\R[1]\neq (\Po_\R[1]U)^*$, since $\psi\in \dom((\Po_\R[1]U)^*)$ does not even
have to be continuous outside $[0,\ell]$.
Instead, we have $\Po_\R[1](\dom(\Po_\R[1])\cap U(L^2(\cal I)))\subseteq U(L^2(\cal I))$, because if $\vp\in L^2(\R)$ vanishes outside $[0,\ell]$, then  
$(\Po_\R[1]\vp)(x)=0$ for $x\notin [0,\ell]$.

Hence,
we know from Proposition \ref{trivialprop} (b) that
${\Po}'[2]=\tilde{\Po}[1]^*\tilde{\Po}[1] =P_0^*P_0.$
The domain of this operator is characterized by the boundary condition $\psi(0)=\psi(\ell)=0$, and the requirements that $\psi$ be continuously differentiable and
$\psi''\in L^2(\cal I)$. Note that this operator is selfadjoint by Proposition \ref{propotatsioone10}; the operator $(2m)^{-1}{\Po}'[2]$ is the Hamiltonian operator for the particle of mass $m$ confined to move in the interval $\cal I$ (``particle in a box''). The spectrum of this operator is discrete and has the complete orthonormal system of eigenvectors $\psi_n$,
$$
\psi_n(x) = \sqrt{\frac 2 \ell} \sin (n\pi x/\ell),\qquad n\in\Z,\;0\le x\le \ell,
$$
associated with eigenvalues $\lambda_n=n^2\pi^2/(2\ell^2)$, that is,
$$
 {\Po}'[2]\psi=
\sum_{n\in\Z} \l_n \ang{ \psi_n}{\psi} \psi_n,
\qquad \psi\in \dom(\Po'[2])= \SSet{\psi\in L^2(\cal I)}{\sum_{n\in\Z} \l_n^2
|\langle \psi_n|\psi\rangle|^2<\infty}.
$$
We will show in the Appendix that $\tilde{\Po}[2]=\Po[2]=P_0^2$. As required, this a restriction of ${\Po}'[2]=P_0^*P_0$, and the difference is exactly in the additional boundary condition $\psi'(0)=\psi'(\ell) = 0$ for any $\psi\in\dom(P_0^2)\subset \dom(P_0^*P_0)$.

To conclude, the physically reasonable definition for the second moment operator of the POVM $\E$ is the symmetric weak operator integral ${\Po}'[2]$ rather than the strong operator integral $\tilde{\Po}[2]=\Po[2]$. With this choice, the POVM $\E$ satisfies analogous ``quantization rules" as the full line momentum, up to second moments.

\section*{Appendix}

The notation
$\varphi^{( k)}=d^k\varphi/dx^k$, $\varphi^{(0)}=\varphi$, will be
used.
Define, for all $n=1,2,...,$  the Sobolev-Hilbert spaces
\begin{equation*}
H^n(\mathcal{I})=\SSet{ \varphi\in C^{n-1}(\mathcal{I})}{
\varphi^{(n-1)} \text{ is absolutely continuous and } \varphi^{(n)} \in
L^2(\mathcal{I}) }
\end{equation*}
where $C^k(\mathcal{I})$ is the
space of $k$-times continuously differentiable complex functions
on $\mathcal{I}$ (and $C^0(\mathcal{I})$ stands for continuous functions).
For $n=1$ we write $H^1(\mathcal{I})=H(\mathcal{I})$.

We start with the definition for the moment operators that usually
appears in the literature, namely
$\tilde{\mathsf{\Po}}[n]$, $n\in\N$.
The following result
was briefly mentioned by Werner \cite{Werner}. We give a proof
here in order to emphasize that care has to be taken on absolute continuity. That care is needed can also be
deduced from the fact that the only difference between the
integrals that one has as a tool is their domains.

\begin{proposition}\label{squaremoments} $\tilde\Po[n] =P_0^n$, and $\Po $ is variance-free.
\end{proposition}
\begin{proof}
According to the definition, the square integrability domain is
$$
\dom(\tilde\Po[n]):= \SSet{\varphi\in
L^2(\mathcal{I})}{ \int x^{2n}\,
d{\Po }_{\varphi,\varphi}(x)<\infty}.
$$
Since $\ang{ \varphi }{\Po (X)\varphi} =
\ang{
{U}\varphi}{\Po_\R(X){U}\varphi}$
for $\varphi\in L^2(\mathcal{I})$, it follows immediately from the
usual spectral theory that
${U}\dom(\tilde\Po[n])=
\dom\big(\Po_\R[n])\cap {U}(L^2(\mathcal{I})\big)$. Each function
$\varphi:\mathbb{R}\to \mathbb{C}$ belonging to $\dom(\Po_\R[1])$ is
absolutely continuous, so it follows
that
\begin{align*}
\dom(\tilde\Po[n])=\SSet{ \varphi \in H^n(\mathcal{I})}{ \frac{d^k}{dx^k}\varphi(0)
=\frac{d^k}{dx^k}\varphi(\ell)=0 \text{ for } k=0,1,\ldots, n-1
}=\dom\rd{P_0^n}.
\end{align*}
Then given a $\varphi\in \dom\rd{P_0^n}$,
$\Po_{\psi,\varphi}(X)=\ang{
{U}\psi}{\Po_\R(X){U}\varphi}$ for all
$\psi\in L^2(\mathcal{I})$, so by the spectral theorem,
$\<{\psi}|{\tilde\Po[n]\varphi}\> =
\ang{ {U}\psi }{ \Po_\R[n]{U}\varphi}$. The
important point now is that the range of ${U}$ is
stable under $\Po_\R[n]$, i.e.\ $\Po_\R[n]\big(\dom(\Po_\R[n])\,\cap\,
\ran\,{U}\big)\subseteq \ran\,{U}$, since
$\Po_\R[n]$ is a derivative and the functions in the range of
${U}$ vanish outside the interval $\mathcal{I}$. It
follows that $\Po_\R[n]{U}\varphi$ is orthogonal to
$\rd{\ran\,{U}}^\perp$, which implies that $\< \Psi|{{U}\tilde\Po[n]\varphi}\>
= \ang{\Psi}{\Po_\R[n]{U}\varphi}$ for \emph{any} $\Psi\in
L^2(\mathbb{R})$, and so
${U}\tilde\Po[n]\varphi=\Po_\R[n]{U}\varphi$.
Since $\Po_\R[n]$ acts as the differential operator, this clearly
implies that $\tilde\Po[n]$ does the
same. Hence, $\tilde\Po[n] =P_0^n$.
The fact that $\Po $ is variance-free
follows from the relation
${U}\tilde\Po[n]\varphi=\Po_\R[n]{U}\varphi$
(see \cite{Werner}). As the proof is very short, we give it here:
\begin{align*}
\M{\tilde\Po[1]\varphi}^2 =
\M{{U}\tilde\Po[1]\varphi}^2
= \M{\Po_\R[n]{U}\varphi}^2 = \int x^{2n}\,
d[{\Po_\R}]_{{U}\varphi,{U}\varphi}(x)
= \int
x^{2n}\, d\Po_{\varphi,\varphi}(x)
\end{align*}
for each $\varphi\in \dom(\tilde\Po[1])$.
\end{proof}

We now proceed to the other two definitions $\Po[n]$ and $\Po'[n]$. The first thing to
note is that both the strong operator integral
$\Po [n]$ and the weak one
$\Po'[n]$ are symmetric extensions
of $\tilde\Po[n]=P_0^n$.
Hence, it follows that
$\dom(\Po [n])\subseteq
\dom(\Po'[n])\subseteq
H^n(\mathcal{I})$, and these operators just act as $(-i)^n
{d^n}/{dx^n}$ on their respective domains.

We will first show that $\Po [n]=P_0^n$, for
all $n=1,2,\ldots$ (see Proposition \ref{moments} below.) The
following two lemmas are needed.
Let
$\mathcal{F}:L^2(\mathbb{R})\to L^2(\mathbb{R})$ be the
Fourier-Plancherel operator.\footnote{Here we want to
apply $\mathcal{F}$ to functions in $L^2(\mathcal{I})$. In our
notation this would be written as $\mathcal{F} {U}$;
in order to simplify the notations, we will write
 $\mathcal{F}$ instead.}
Hence,
$$
(\mathcal{F}\varphi)(x) = \frac{1}{\sqrt{2\pi}} \int_0^\ell
e^{-ixt} \varphi(t)\, dt, \; \; x\in \mathbb{R}, \; \varphi\in
L^2(\mathcal{I}).
$$
(Note that every element of $L^2(\mathcal{I})$ is integrable by
the Cauchy-Schwarz inequality.) The function
$\mathcal{F}\varphi:\mathbb{R}\to \mathbb{C}$ is continuous,
bounded, and belongs to $L^2(\mathbb{R})$.
\begin{lemma}\label{apulemma1}
\begin{itemize}
\item[(a)] For any $\varphi\in H(\mathcal{I})$, we have
$$
[\mathcal{F}\varphi'](x) =
ix[\mathcal{F}\varphi](x)+\frac{1}{\sqrt{2\pi}}[\varphi(\ell
)e^{-ix\ell}-\varphi(0)], \ \ x\in \mathbb{R}
$$
\item[(b)] For any $\varphi\in H^n(I)\cap \dom(P_0^{n-1})$, we have
$$
x^n[\mathcal{F}\varphi](x) =
(-i)^n[\mathcal{F}\varphi^{(n)}](x)+\frac{i^n}{\sqrt{2\pi}}[\varphi^{(n-1)}(\ell)e^{-i\ell
x}-\varphi^{(n-1)}(0)], \ \ x\in \mathbb{R}.
$$
\end{itemize}
\end{lemma}
\begin{proof} Straightforward application of absolute continuity and integration-by-parts.\end{proof}

\begin{lemma}\label{apulemma2}
Let $a,b\in \mathbb{C}$. Then $x\mapsto
F_\psi(x):=\overline{[\mathcal{F}\psi](x)}[ae^{-i\ell x}-b]$ is
Lebesgue-integrable over $\mathbb{R}$ for all $\psi\in
L^2(\mathcal{I})$, if and only if $a=b=0$.
\end{lemma}
\begin{proof}
We only need to consider functions $\psi_\theta\in
L^2(\mathcal{I})$, where $\psi_\theta(t):=e^{-i\theta t}$, with
$\theta\in \mathbb{R}$. Then $$ F_{\psi_\theta}(x)=\frac{\big[e^{i
(x+\theta) \ell}-1\big]\big[e^{-i x \ell}a-b\big]}{i(x+\theta)
\sqrt{2\pi}}. $$ If $|a|\neq |b|$, then $|F_{\psi_0}(x)|\geq 2|\sin
(x\ell /2)| \alpha/(i |x| \sqrt{2\pi})$, with $\alpha = \big||a|-
|b|\big|$, so $F_{\psi_0}$ is not integrable. If $|a|=|b|$, take
$\theta$ so that $a =-e^{-i\theta \ell}b$. Then $
F_{\varphi_\theta}(x) =
-2b\sin((x+\theta)\ell)/(\sqrt{2\pi}(x+\theta))$, which is again
not integrable. The only remaining possibility is $a=b=0$, and
then $F_\psi=0$ is trivially integrable.
\end{proof}

\begin{proposition}\label{moments} $\Po [n] =\tilde\Po[n] = P_0^n$ for all $n=1,2,\ldots$.
\end{proposition}
\begin{proof} We have already noted that $\Po [n]$ is a restriction of 
$(-i)^n
{d^n}/{dx^n}:H^n(\mathcal{I})\to L^2(\mathcal{I})$. In particular,
$\dom(\Po [n])\subseteq H^n(\mathcal{I})$.
Thus, we only need to show that the vectors in
$\dom(\Po [n])$ are exactly those elements
of $H^n(\cal I)$ which satisfy the boundary conditions defining
$\dom(P_0^n)$. Proceeding by induction, we first consider the case
$n=1$. Using Lemma \ref{apulemma1}, we get
\begin{equation}\label{firststep}
x \overline{[\mathcal{F}\psi](x)}[\mathcal{F}\varphi](x) = -i
\overline{[\mathcal{F}\psi](x)}[\mathcal{F}\varphi'](x)+\frac{i}{\sqrt{2\pi}}\overline{[\mathcal{F}\psi](x)}[\varphi(\ell)e^{-i\ell
x}-\varphi(0)],
\end{equation}
for $\psi\in L^2(\mathcal{I})$ and $\varphi\in H(\mathcal{I})$. By
definition, $\varphi\in \dom(\Po [1])$ if
and only if $x\mapsto
\overline{[\mathcal{F}\psi](x)}[\mathcal{F}\varphi](x)$ is
integrable over $\mathbb{R}$ for all $\psi\in L^2(\mathcal{I})$.
Since $\varphi'\in L^2(\mathcal{I})$, so that both
$\mathcal{F}\psi$ and $\mathcal{F}\varphi'$ are in
$L^2(\mathbb{R})$, the first term in the right hand side of
\eqref{firststep} is integrable in any case. Hence $\varphi\in
\dom(\Po [1])$ if and only if the second
term is integrable for all $\psi\in L^2(\mathcal{I})$. But by
Lemma \ref{apulemma2}, this happens exactly when
$\varphi(0)=\varphi(\ell)=0$, i.e.\ $\varphi\in \dom(P_0)$. Thus,
$\Po [1]=P_0$. Now we assume inductively
that $\Po [n-1]=P_0^{n-1}$. Since
$|x^{n-1}|\leq 1+|x^n|$ for all $x\in \mathbb{R}$, and the
relevant complex measures are finite, it follows that
$\Po [n]\subseteq
\Po [n-1]=P_0^{n-1}$, where the last
equality follows from the induction assumption. Hence,
$\dom(\Po [n])\subseteq H^n(\mathcal{I})\cap
\dom(P_0^{n-1})$. Letting $\varphi\in H^n(\mathcal{I})\cap
\dom(P_0^{n-1})$ we get from Lemma \ref{apulemma1} (b) that

\begin{align*}
x^n \overline{[\mathcal{F}\psi](x)}[\mathcal{F}\varphi](x) =
(-i)^n
\overline{[\mathcal{F}\psi](x)}[\mathcal{F}\varphi^{(n)}](x)
+\frac{i^n}{\sqrt{2\pi}}\overline{[\mathcal{F}\psi](x)}[\varphi^{(n-1)}(\ell)e^{-i\ell
x}-\varphi^{(n-1)}(0)]
\end{align*}

for all $\psi\in L^2(\mathcal{I})$. Since now $\varphi^{(n)}\in
L^2(\mathcal{I})$ (because $\varphi\in H^n(\mathcal{I})$), we can
again use the same argument as before to conclude by Lemma
\ref{apulemma2} that $\varphi\in
\dom(\Po [n])$ if and only if
$\varphi^{(n-1)}(\ell )=\varphi^{(n-1)}(0)=0$, i.e.\ $\varphi\in
\dom(P_0^n)$. The proof is complete.
\end{proof}
For the weak operator integral
$\Po'[n]$, the following result
holds.

\begin{proposition}\label{weaks} $\Po'[1]=P_0$ and $\Po'[2n]=(P_0^n)^*P_0^n$.
\end{proposition}
\begin{proof} The last statement follows from Propositions \ref{trivialprop} (b) and \ref{squaremoments}, because the derivative of a function with support in $\cal I$ also has support in $\cal I$ (this was already mentioned in the proof of Proposition \ref{squaremoments}).

To prove that $\Po'[1]=P_0$,
recall first that $\Po'[1]$ is a
symmetric extension of
$\tilde\Po[1]$, and hence coincides
with one of the selfadjoint extensions\footnote{Selfadjoint extensions of $P_0$ are $-id/dx$ on the domains given by the boundary conditions $\psi(\ell)=e^{i\theta}\psi(0)$.} $P^{(\theta)}$ of $P_0$, or $P_0$
itself. We show that
$\dom(\Po'[1])$ is a proper
subspace of $\dom(P^{(\theta)})$, so that
$\Po'[1]=P_0$ must hold. For any
$a,b\in \mathbb{C}$, define
$\varphi_{a,b}:\,\mathbb{R}\to\mathbb{C}$ via
$\varphi_{a}^{b}(t):=(b-a)t/\ell+a$. This is obviously infinitely
differentiable, and satisfies the boundary conditions
$\varphi_{a,b}(0)=a$ and $\varphi_{a,b}(\ell)=b$, so for a
suitable choice of the two constants, the vector $\varphi_{a,b}$
will belong to the domain of a given $P^{(\theta)}$. We will show
that it does not belong to the form domain
$D_0(x,\Po )$ (which is even larger than the
domain of $\Po'[1]$), unless
$a=b=0$. In order to prove this, it suffices to show that $xG(x)$
is not integrable over $[1,\infty)$, where $G:\mathbb{R}\to
\mathbb{C}$ is the density of the measure
$\Po _{\varphi_{a,b},\varphi_{a,b}}$, i.e.
$G(x):= |(\mathcal{F}\varphi_{a,b})(x)|^2 =
(2\pi)^{-1}\left|\int_0^\ell e^{-ixt}\varphi_{a,b}(t)dt\right|^2$.
Now in case $a=b\neq 0$, we have simply $xG(x) =
2|a|^2[1-\cos(x\ell)]/(2\pi x)$, which is not integrable. In case
$a\neq b$, we put $a':= (b-a) \ell^{-1}\neq 0$, $b':=a/a'$; then
we get $2\pi|a|^{-2} xG(x) = h(x) +x^{-2}(f(x)+x^{-1} g(x))$,
where $h(x) :=x^{-1} |\ell+b'-b' e^{i x \ell}|^2$, and $f$ and $g$
are bounded real functions. Now $\int xG(x)\,  dx =\infty$ is
equivalent to $\int_1^\infty h(x) \, dx =\infty$, which is true
because $h(x) \ge\frac{(|\ell+b|-|b|)^2}{x}$ in case $|\ell+b|\neq
|b|$, while $h(x)=2|b|^2x^{-1}[1-\cos(x\ell+\beta)]$ for some
$\beta\in [0,2\pi)$ in case $|\ell+b|=|b|$. The proof is complete.
\end{proof}

\begin{remark}\rm It is interesting to compare the domains of the differential operators
$\tilde\Po[2n]=\Po [2n]$
and $\Po'[2n]$, both acting as
restrictions of the maximal operator $(-1)^n{d^{2n}}/{dx^{2n}}$,
and thereby differing only by boundary conditions. Explicitly, we
have

\begin{eqnarray*}
\dom(\Po [2n])&=& \dom(P_0^{2n}) = \SSet{
\varphi \in H^{2n}(\mathcal{I})}{
\varphi^{(k)}(0)=\varphi^{(k)}(\ell)=0,\; k=0,1,\ldots, 2n-1};\\
\dom(\Po'[2n])&=&\dom
((P_0^n)^*P_0^n) = \SSet{ \varphi \in H^{2n}(\mathcal{I})}{
\varphi^{(k)}(0)=\varphi^{(k)}(\ell)=0,\; k=0,\ldots, n-1}.
\end{eqnarray*}

(To obtain the last equality, recall that $\dom((P_0^n)^*) =
H^{n}(\mathcal{I})$.) Hence, in the case of even index, the weak
moment operator integral differs from the strong one in that half
of the boundary conditions are removed. Note also that
$\Po'[2n]$ is selfadjoint, because
$P_0^n$ is closed. However, as the example
$\Po'[1]=P_0$ shows, odd moments
need not be.
\end{remark}

\noindent{\bf Acknowledgments.} 
We acknowledge financial support from the CHIST-ERA/BMBF project CQC, and the Academy of Finland grant no 138135.

\end{document}